\documentclass[11pt,reqno]{amsart}

\usepackage{a4wide}
\usepackage{amsmath,amssymb}
\usepackage{bbm}
\usepackage{xcolor}
\usepackage{graphicx}
\usepackage{ae,aecompl}
\usepackage{textcomp}

\usepackage{scalerel,stackengine}
\stackMath
\newcommand\reallywidehat[1]{%
\savestack{\tmpbox}{\stretchto{%
  \scaleto{%
    \scalerel*[\widthof{\ensuremath{#1}}]{\kern-.6pt\bigwedge\kern-.6pt}%
    {\rule[-\textheight/2]{1ex}{\textheight}}%WIDTH-LIMITED BIG WEDGE
  }{\textheight}%
}{0.5ex}}%
\stackon[1pt]{#1}{\tmpbox}%
}

\newcommand{\ver}{{\vert\kern-0.25ex\vert\kern-0.25ex\vert }}

\allowdisplaybreaks

\theoremstyle{plain}
\newtheorem{theorem}{Theorem}[section]
\newtheorem{prop}[theorem]{Proposition}
\newtheorem{lemma}[theorem]{Lemma}
\newtheorem{coro}[theorem]{Corollary}

\theoremstyle{definition}
\newtheorem{definition}[theorem]{Definition}
\newtheorem{remark}[theorem]{Remark}
\newtheorem{example}[theorem]{Example}

\newcommand{\ZZ}{{\mathbb Z}}

\newcommand{\RR}{{\mathbb R}}
\newcommand{\NN}{{\mathbb N}}
\newcommand{\CC}{{\mathbb C}}

\newcommand{\cA}{{\mathcal A}}
\newcommand{\cB}{{\mathcal B}}

\newcommand{\im}{{\mathrm{i}}}
\newcommand{\supp}{{\mbox{supp}}}

\newcommand{\mc}{\mathcal}

\newcommand{\dd}{\mbox{d}}

\newcommand{\eps}{\varepsilon}
\newcommand{\cM}{{\mathcal M}}

\newcommand{\WAP}{\mathcal{W}\hspace*{-1pt}\mathcal{AP}}

\newcommand{\SAP}{\mathcal{S}\hspace*{-2pt}\mathcal{AP}}

\newcommand{\Bap}{\mathcal{B}\hspace*{-1pt}{\mathsf{ap}}}

\newcommand{\bap}{Bap}

\newcommand{\exend}{\hfill $\Diamond$}
\newcommand{\lm}{\ensuremath{\lambda\!\!\!\lambda}}

\newcommand{\ts}{\hspace{0.5mm}}

\newcommand{\leftbrac}{\text{\textlquill} }
\newcommand{\rightbrac}{\text{\textrquill} }

\newcommand{\Cu}{C_{\mathsf{u}}}
\newcommand{\Cc}{C_{\mathsf{c}}}

\makeatletter
\newcommand{\ostar}{\mathbin{\mathpalette\make@circled\star}}
\newcommand{\make@circled}[2]{%
  \ooalign{$\m@th#1\smallbigcirc{#1}$\cr\hidewidth$\m@th#1#2$\hidewidth\cr}%
}
\newcommand{\smallbigcirc}[1]{%
  \vcenter{\hbox{\scalebox{0.77778}{$\m@th#1\bigcirc$}}}%
}
\makeatother

\newcommand{\ebeA}{\circledast_{\cA}}
\newcommand{\ebeB}{\circledast_{\cB}}

\begin{document}

\title{The (twisted) Eberlein convolution of measures}

\dedicatory{Dedicated to our friend, Uwe Grimm.}

\author{Daniel Lenz}
\address{Mathematisches Institut, Friedrich Schiller Universit\"at Jena, 07743 Jena, Germany}
\email{daniel.lenz@uni-jena.de}
\urladdr{http://www.analysis-lenz.uni-jena.de}

\author{Timo Spindeler}
\address{Fakult\"at f\"ur Mathematik, Universit\"at Bielefeld, \newline
\hspace*{\parindent}Postfach 100131, 33501 Bielefeld, Germany}
\email{tspindel@math.uni-bielefeld.de}

\author{Nicolae Strungaru}
\address{Department of Mathematical Sciences, MacEwan University \\
10700 -- 104 Avenue, Edmonton, AB, T5J 4S2, Canada\\
and \\
Institute of Mathematics ``Simon Stoilow''\\
Bucharest, Romania}
\email{strungarun@macewan.ca}
\urladdr{http://academic.macewan.ca/strungarun/}

\begin{abstract}
In this paper, we study the properties of the Eberlein convolution
of measures and introduce a twisted version of it. For functions we
show that the twisted Eberlein convolution can be seen as a
translation invariant function-valued inner product. We study its
regularity properties and show its existence on suitable sets of
functions. For translation bounded measures we show that the
(twisted) Eberlein convolution  always exists along subsequences of
the given sequence, and is a weakly almost periodic and Fourier
transformable measure.  We prove that if one of the two measures is
mean almost periodic, then the (twisted) Eberlein convolution is
strongly almost periodic. Moreover, if one of the measures is norm
almost periodic, so is the (twisted) Eberlein convolution.
\end{abstract}

\keywords{Almost periodic measures, Eberlein convolution, pure point
diffraction, spectral theory}

\subjclass[2010]{52C23, 37A45, 42A75, 43A05, 43A60, 43A07, 43A25,
37A30}

\maketitle

\section{Introduction}
Quasicrystals are structures that are ordered despite not being
periodic. A key feature is occurrence of (pure) point diffraction.
Quasicrystals were discovered in 1982 \cite{She} and have since then
been studied by physicists, material scientists, chemists and
mathematicians.

There are two important  classes of models used in the  study of
quasicrystals, namely cut and project schemes and substitution
rules. For cut and project sets, occurrence of pure point spectrum is
well understood. For substitutions there still is an open question
in this direction  known as Pisot conjecture.  This remains one of
the most important open question in the field of aperiodic order.

A recent approach centered around renormalisation equations
\cite{BGM} shows promise for the study of (Pisot) substitution and
some progress in this direction has been made, see for example
\cite{BFGR,BG,BGM} just to name a few.

The basic idea beyond the renormalisation approach is the following.
Consider a (Pisot) substitution on a finite alphabet $\{a_1,\ldots ,
a_k \}$ and let $\Lambda_1, \ldots, \Lambda_k$ be the left-end
points of each tile-type in the geometric fixed point of the
substitution. For each $1 \leqslant i,j \leqslant k$, define the pair
correlations
\[
\gamma_{ij}=\delta_{\Lambda_i} \ebeA \widetilde{\delta_{\Lambda_j}}
\ts.
\]
Then, the inflation rule for the substitution induces a system of
equations, called the renormalisation equations, on the pair
correlation measures. Via the study of the renormalisation
equations, the spectral nature of various substitution has been
established in the papers mentioned above.

This approach emphasizes the need to study Eberlein convolutions of
the type $\mu \ebeA \widetilde{\nu}$. It is the primary goal of this
article to systematically study such Eberlein convolutions of
arbitrary measures in full generality.

Let us emphasize here that, building on the work of Baake--Grimm
\cite{BG2,BG3}, an orthogonality type relation between measures of
different spectral type with respect to this type of Eberlein
convolutions has been shown in \cite{BS,NS21b}. This seems to be the
first step towards a generalized Eberlein decomposition for the
original structure itself, which would simplify the study of
diffraction measure for systems with mixed spectra.

One of the issues with respect to the Eberlein convolution is its
behaviour with respect to commutativity. The issue becomes clear in
the equation
\[
\mu \ebeA \nu =\nu \circledast_{-\cA} \mu
\]
which shows that commutativity does not hold unless the van Hove
sequence $\cA$ is symmetric. This relation can create issues when
one deals with the Fourier--Bohr coefficients of such convolutions,
as well as with the autocorrelation of linear combinations of
measures. It turns out that for computations involving the
autocorrelation, it is useful to consider a slightly different
version of the Eberlein convolution. Indeed, in all such
computations one deals only with Eberlein convolutions of the form
$\mu \ebeA \widetilde{\nu}$ which satisfy the $\sim$-commutativity
relation
\[
\nu \ebeA \widetilde{\mu} = \widetilde{ \mu \ebeA \widetilde{\nu} }
\ts.
\]
To avoid the (temporary) switch in the van Hove sequence, which, if
treated with care, disappears at the end of the computation, we
introduce the notion of twisted Eberlein convolution as
\[
\leftbrac \mu , \nu \rightbrac_{\cA} := \mu \ebeA \widetilde{\nu}
\ts.
\]
This allows us to always work with the same van Hove sequence when
dealing with this type of computations, and transform the reflection
to the $\sim$-commutativity relation
\[
\leftbrac \nu , \mu \rightbrac_{\cA} =\widetilde{\leftbrac \mu , \nu
\rightbrac_{\cA}} \ts.
\]
We show that the (twisted) Eberlein convolution of measures always
exists along a subsequence of the given van Hove sequence, and it is
a Fourier transformable measure that is weakly almost periodic.
Moreover, if one of the given measures $\mu, \nu$ is mean almost
periodic with respect to $\cA$, we prove in Theorem~\ref{thm tebe is
SAP} that $\leftbrac \mu , \nu \rightbrac_{\cA}$ is a strongly
almost periodic measure. In particular, we get an alternate proof
that a mean almost periodic measure has pure point diffraction
spectrum. Also, given two translation bounded and Besicovitch almost
periodic measures $\mu,\nu \in \Bap_{\cA}^2(G)$, we show that their
twisted Eberlein convolution $\leftbrac \mu , \nu \rightbrac_{\cA}$
exists along $\cA$, has pure point Fourier transform and satisfies
the generalized Consistent Phase Property (CPP)
\[
\reallywidehat{\leftbrac \mu , \nu \rightbrac_{\cA}}(\{ \chi
\})=a_\chi^{\cA}(\mu) \overline{a_\chi^{\cA}(\nu) } \ts.
\]

On a more structural level  our results can be understood as
follows:  On the strongly almost periodic functions $SAP(G)$ there
exists a unique map
\[
\langle \cdot , \cdot \rangle : SAP(G)\times SAP(G) \to SAP(G)
\]
with the following properties:
\begin{itemize}
\item The map is  compatible with translates: $T_t \langle f,g \rangle =
\langle  f, T_t g \rangle = \langle T_{-t} f, g\rangle $ for all
$t\in G$.
\item  The map is linear in the first argument and conjugate linear in the
second argument.
\item  The function $\langle f, f\rangle $ is positive definite.
\item  $\langle \chi,\chi\rangle = \chi$ for any character $\chi \in
\widehat{G}$.
\end{itemize}
Observe here that for any two characters $\chi,\psi \in \widehat{G}$
and all $t \in G$ we have
\[
\langle \chi, \psi \rangle= \langle T_t \chi , T_t \psi \rangle =
\chi(t) \overline{\psi(t)} \langle \chi, \psi \rangle \,,
\]
which immediately implies the orthogonality relation $\langle \chi,
\psi \rangle=0$ for all $\chi \neq \psi$.

It is easy to see that one such a map is given by
\[
\langle f,g \rangle (0) = M(f\bar{g})
\]
for all $f,g$. Moreover, from positive definiteness and translation
invariance we infer that any such  must have the continuity feature
that
\[
\| \langle f,g \rangle \|_\infty \leqslant  \| f \|_{b,2,\cA} \| g
\|_{b,2
  ,\cA} \ts.
\]
This in turn allows to extend the map from the characters to linear
combination of characters and uniqueness follows. The main thrust of
the article is then to extend such a map to larger classes of
functions and even measures. It turns out that such an extension is
possible and the map will still have its range in the strongly
almost periodic functions provided a suitable weak smoothness of its
factors holds.

\section{Preliminaries}
Throughout this paper, $G$ is a second countable, locally compact
(Hausdorff) Abelian group. Its corresponding Haar measure is denoted
by $\theta_G$ or simply $|\cdot|$. Integration with respect to the
Haar measure is denoted by $\int \ts \dd t$ or $\int \ts \dd s$. As
usual, the space of continuous and compactly supported functions
from $G$ to $\CC$ is denoted by $\Cc(G)$, while the space of
uniformly continuous and bounded functions is denoted by $\Cu(G)$.
For $p\in [1,\infty]$, we denote by $L^p (G)$ the space of
(equivalence classes of) $p$-integrable functions and by
$L^p_{\text{loc}}(G)$ the space of (equivalence classes of) functions $f$
whose restrictions to compact set belong to  $L^p (G)$.

For any function $g$ on $G$, we define the functions $\tau_tg$ (with $t\in G$), $\widetilde{g}$ and $g^{\dagger}$ by
\[
(\tau_tg)(x):= g(x-t) \ts, \qquad \widetilde{g}(x) :=
\overline{g(-x)} \qquad \text{ and } \qquad g^{\dagger}(x):= g(-x)
\ts.
\]
Likewise, we define for any measure $\mu$ on $G$ the measures $\tau_t\mu$, $\widetilde{\mu}$ and $\mu^{\dagger}$ by
\[
(\tau_t\mu)(g):= \mu(\tau_{-t}g) \ts, \qquad \widetilde{\mu}(g) := \overline{\mu(\widetilde{g})} \qquad \text{ and } \qquad \mu^{\dagger}(g):= \mu(g^{\dagger}) \ts.
\]
Moreover, for a measure $\mu$ and function $\varphi\in \Cc (G)$, we define
the function $\mu\ast \varphi$ on $G$ by
\[
(\mu\ast \varphi)(t) =(\tau_{-t} \mu) (g^{\dagger}) \ts.
\]
For a measurable function $f$ on $G$ that is locally in $L^1(G)$, we
define $f\ast\varphi:=(f \theta_G) \ast\varphi$. Note that for us
a measure is a linear functional $\mu:\Cc(G)\to\CC$ such that, for
every compact set $K\subseteq G$, there is a constant $c_K>0$ with
\[
|\mu(\varphi)| \leqslant c_K \ts \|\varphi\|_{\infty}
\]
for every $\varphi\in\Cc(G)$ with $\operatorname{supp}(\varphi)\subseteq K$. As
usual, here  $\|\varphi\|_{\infty}:=\sup_{t\in G} |\varphi(t)|$ denotes the
supremum norm. Due to Riesz' representation theorem \cite{rud2}, our
concept of measures coincides with the concept of regular Radon
measures from measure theory. A measure $\mu$ is called positive if
$\mu(\varphi)\geqslant 0$ for all $\varphi \in \Cc (G)$ with $\varphi
\geqslant 0$. For every measure $\mu$ on $G$, there is a  unique positive
measure $|\mu|$ such that
\[
|\mu|(\varphi) = \sup\{ |\mu(\psi)| \ts :\ts \psi\in\Cc(G),\ts |\psi|\leqslant \varphi\}
\]
for every $\varphi\in\Cc(G)$ with $\varphi\geqslant0$, see \cite{Ped}. A measure
$\mu$ is \textit{translation bounded} if $\mu\ast \varphi$ is bounded
for any $\varphi \in \Cc (G)$. Equivalently, a measure $\mu$
is translation bounded if and only if
for any compact $K$ in $G$ with non-empty interior
\[
 \| \mu \|_{K}:=\sup \{|\mu|(t+K) : t\in G\}<\infty
\]
holds. Note that then $\mu \ast \varphi$ must belong to $\Cu (G)$
for any $\varphi \in \Cc (G)$. This is crucial as it will allow us
to consider and study translation bounded measures as - so to speak
- dual objects to $\Cu (G)$.  The set of all translation bounded
measures on $G$ is denoted by $\cM^\infty (G)$.

\begin{definition}[F{\o}lner and van Hove sequences]
A sequence $(A_n)_{n\in\NN}$ of precompact Borel subsets of $G$ of
positive measure is called a \textit{F{\o}lner sequence} sequence
if
\[
\lim_{n\to\infty} \frac{| A_n\ts \Delta\ts (t+A_n)|}{|A_n|}=0
\]
for all $t \in G$.
A sequence $(A_n)_{n\in\NN}$ of precompact open subsets of $G$ is called a \textit{van Hove} sequence if
\[
     \lim_{n\to\infty} \frac{|\partial^{K} A_{n}|}
     {|A_{n}|}  =  0 \ts ,
\]
for each compact set $K \subseteq G$, where the \textit{$K$-boundary $\partial^{K} A$} of an open set $A$ is defined as
\[
\partial^{K} A := \bigl( \overline{A+K}\setminus A\bigr) \cup
\bigl((\left(G \backslash A\right) - K)\cap \overline{A}\ts \bigr) \ts.
\]
\end{definition}

\begin{remark}
For our considerations below we will need to use van
Hove sequences (rather than just F\o lner sequences) in order to
establish the vanishing of boundary terms after cut-off. This will not
restrict the applicability of our theory as it  is known that every
$\sigma$-compact locally compact Abelian group $G$ admits van Hove
sequences \cite{Martin2}.  \exend
\end{remark}

\medskip

In the spirit of \cite{ARMA}, for $\varphi \in \Cc(G)$ and $\mu \in
\cM^\infty(G)$, we define
\[
\| \mu \|_{\varphi}:= \| \mu*\varphi \|_\infty \ts.
\]
The topology defined by this family of semi-norms on $\cM^\infty(G)$ is called the \textit{product topology for measures} and will be denoted by $\tau_{\operatorname{p}}$.
In this topology, a net $(\mu_\alpha)_{\alpha}$ converges to $\mu$ if and only if, for all $\varphi \in \Cc(G)$, the net $(\mu_\alpha*\varphi)_{\alpha}$ converges to $\mu*\varphi$ in $(\Cu(G), \| \cdot\|_\infty)$.
We will refer to the topology defined by the family of semi-norms $\{
\| \cdot \|_{\varphi*\psi} : \varphi, \psi \in \Cc(G)\}$ as the
\textit{semi-product topology for measures} and will denoted this
topology by $\tau_{\operatorname{sp}}$. This topology was introduced
and studied in \cite{LSS4}.

At the end of this section, let us review some common notions of
almost periodic functions and measures. Ultimately, these notions
are all stemming from the supremum norm $\|\cdot\|_\infty $ on $\Cu
(G)$.

\begin{definition}[Almost periodicity notions based on $(\Cu(G),\|\cdot\|_\infty)$]
A function $f\in \Cu(G)$ is called \textit{weakly (strongly) almost
periodic} if the closure of the set $\{\tau_tf\ts :\ts t\in G\}$ in
the weak (norm) topology of the Banach space  $(\Cu(G),\|\cdot\|_\infty)$ is compact. The space of weakly
(strongly) almost periodic functions is denoted by $W\!AP(G)$
($SAP(G)$).

A measure $\mu$ on $G$ is called \textit{weakly (strongly) almost
periodic} if $\mu*\varphi$ is a weakly (strongly) almost periodic
function for all $\varphi\in\Cc(G)$. We will denote the space of
weakly (strongly) almost periodic measures by $\mathcal{W\!AP}(G)$
($\mathcal{S\!AP}(G)$).

Let $K$ be a compact subset of $G$ that has a non-empty interior. A
measure $\mu$ is called \textit{norm almost periodic} if, for all
$\eps>0$, the set
\[
\{t\in G\ts :\ts \|\mu-\tau_t\mu\|_K <\eps\}
\]
is relatively dense in $G$.
\end{definition}

\subsection{The Besicovitch semi-norm and induced notions of almost
periodicity} By replacing the supremum norm $\|\cdot\|_\infty$ on
$\Cu (G)$  by other (semi-)norms we obtain further relevant notions
of almost periodicity. One of these is discussed in this section.

\medskip

Let $1\leqslant p <\infty$, and let $\mathcal{A}=(A_n)_{n\in\NN}$ be a van Hove sequence.
For $f\in L^p_{\text{loc}}(G)$, we define
\[
\|f\|_{b,p,\mathcal{A}} := \left(\limsup_{n\to\infty}
\frac{1}{|A_n|} \int_{A_n} |f(t)|^p \ \dd t \right)^{\frac{1}{p}}\in
[0,\infty] \ts.
\]
This defines a semi-norm on the space
\[
BL_{\cA}^p(G):=\{f\in L_{\text{loc}}^p(G)\ts :\ts
\|f\|_{b,p,\cA}<\infty\} \ts.
\]

This space has the following completeness property.

\begin{theorem}\cite[Thm.~1.18]{LSS} For $1 \leqslant p < \infty$, the space $(BL_{\cA}^p(G), \| \cdot \|_{b,p,\cA})$ is complete.  \qed
\end{theorem}

The norm $\|\cdot\|_{b,p,\cA}$ is not  invariant under translations.
To remedy this, we follow \cite{LSS} by mostly working in
$BC_{\cA}^p(G)$, which is  the closure of $\Cu(G)$ in
$BL_{\cA}^p(G)$.

A simple direct computation (see \cite{LSS4}\label{lem:5} as well)
gives the following statement.

\begin{lemma}
Let $f \in \Cu(G)$ and $\varphi \in \Cc(G)$. Then, one has
\[
\| f*\varphi \|_{b,1,\cA}  \leqslant  \|\varphi\|_{1}\| f \|_{b,1,\cA}  \ts.  \tag*{\qed}
\]
\end{lemma}

The semi-norm $\|\cdot\|_{b,1,\cA}$ is needed in order to define mean almost periodic functions.

\begin{definition}[Mean almost periodic functions]\cite{LSS}
Let $\cA$ be a van Hove sequence in $G$. A function $f\in \Cu(G)$ is called \textit{mean almost periodic} (with respect to $\cA$) if, for each $\eps>0$, the set
\[
\{t\in G\ts :\ts \|f-\tau_t f\|_{b,1,\cA}<\eps \}
\]
is relatively dense. We denote the set of mean almost periodic functions by $\operatorname{MAP}_{\cA}(G)$, or simply $\operatorname{MAP}(G)$.
\end{definition}

The set $\operatorname{MAP}(G)$ is stable under convolution with elements from $\Cc(G)$, which immediately follows from Lemma~\ref{lem:5}.

\begin{coro}
Let $f \in \operatorname{MAP}(G)$ and $\varphi \in \Cc(G)$. Then,
$f*\varphi \in \operatorname{MAP}(G)$. \qed
\end{coro}

\begin{definition}[$p$-mean almost periodic functions]
Let $\cA$ be a van Hove sequence, and let $1\leqslant p < \infty$. We denote by $\operatorname{Map}_{\cA}^p(G)$ the closure of $\operatorname{MAP}_{\cA}(G)$ in $BC_{\cA}^p(G)$ with respect to $\|\cdot \|_{b,p,\cA}$. The elements in $\operatorname{Map}_{\cA}^p(G)$ are called \textit{$p$-mean almost periodic}.

A measure $\mu$ on $G$ is called \textit{$p$-mean almost periodic} (with respect to $\cA$) if
\[
\mu*\varphi\in \operatorname{Map}_{\cA}^p(G) \qquad \text{ for all } \varphi\in\Cc(G) \ts.
\]
The space of all \textit{$p$-mean almost periodic} measures is denoted by $\cM\texttt{ap}_{\cA}^p(G)$. For $p=1$, we simply write $\cM\texttt{ap}_{\cA}(G)$ or $\cM\texttt{ap}(G)$.
\end{definition}

\begin{prop}
Let $\mu\in\mathcal{M}^{\infty}(G)$. Then, the following statements
are equivalent:
\begin{itemize}
\item[(i)] $\mu\in\cM\mathtt{ap}_{\cA}(G)$

\item[(ii)]  $\mu*\varphi *\psi \in\operatorname{MAP}_{\cA}(G) \ \text{ for all
} \varphi,\psi \in \Cc (G)$.
\end{itemize}
\end{prop}

\begin{proof}
If $\mu\in\cM\mathtt{ap}_{\cA}(G)$, then $\mu*\varphi*\psi
\in\operatorname{MAP}_{\cA}(G)$ because $\varphi*\psi\in
C_{\text{c}}(G)$.

Conversely, replacing $\psi$ by elements of an approximate identity
$(\phi_{\alpha})_{\alpha}$, we observe that $\{\mu *\varphi *
\phi_{\alpha}\ts :\ts \alpha\}\subseteq \operatorname{MAP}_{\cA}(G)$.
Thus, we have
\[
\mu * \varphi = \lim_{\alpha}\ts (\mu * \varphi * \phi_{\alpha})\in \operatorname{MAP}_{\cA}(G)
\]
by \cite[Thm.~2.8]{LSS}.
\end{proof}

\medskip

As consequence  we obtain the following result.

\begin{prop}
$\cM\mathtt{ap}_{\cA}(G)$ is closed in $\cM^\infty(G)$ with respect to the product topology.
\end{prop}
\begin{proof}
Let $(\mu_\alpha)_{\alpha}$ be a net in $\cM\mathtt{ap}_{\cA}(G)$ which converges in the product topology to some $\mu \in \cM^\infty(G)$.

Since $\mu_\alpha \in \cM\mathtt{ap}_{\cA}(G)$, we have $\mu_\alpha*\varphi \in \operatorname{MAP}_{\cA}(G)$ for all $\varphi \in \Cc(G)$. Therefore, since $\operatorname{MAP}_{\cA}(G)$ is closed in $(\Cu(G), \| \cdot\|_\infty)$, we get $\mu*\varphi \in \operatorname{MAP}_{\cA}(G)$, and hence $\mu \in \cM\mathtt{ap}_{\cA}(G)$.
\end{proof}

\subsection{Fourier--Bohr coefficients}

We start with the definition.

\begin{definition}[Fourier--Bohr coefficients]
Let $\cA$ be a fixed van Hove sequence and $\chi \in \widehat{G}$.
We say that $f \in L^1_{\text{loc}}(G)$ has a well defined
\textit{Fourier--Bohr coefficient} (at $\chi$) with respect to $\cA$
if the following limit exists
\[
a_{\chi}^\cA(f):= \lim_{n\to\infty} \frac{1}{|A_n|} \int_{A_n} f(s)\ts \overline{\chi(s)}\ \dd s \ts.
\]
%Furthermore, we say that the Fourier--Bohr coefficient (at $\chi$)
%exists uniformly if
%\[
%a_{\chi}^\cA(f)= \lim_{n\to\infty} \frac{1}{|A_n|} \int_{t+ A_n}f(s)\ts \overline{\chi(s)}\ \dd s \ts, \qquad \mbox{ uniformly in } t \ts.
%\]
Similarly, we say that a measure $\mu$ has a well defined
\textit{Fourier--Bohr coefficient} (at $\chi$) with respect to $\cA$
if the following limit exists
\[
a_{\chi}^\cA(\mu):= \lim_{n\to\infty} \frac{1}{|A_n|} \int_{A_n} \overline{\chi(s)}\ \dd \mu(s) \ts,
\]
%and that the Fourier--Bohr coefficient exists uniformly if
%\[
%a_{\chi}^\cA(\mu)= \lim_{n\to\infty} \frac{1}{|A_n|} \int_{t+ A_n} \overline{\chi(s)}\ \dd \mu( s) \ts, \qquad \mbox{ uniformly in } t \ts.
%\]
\end{definition}

The Fourier--Bohr coefficients exist for Besicovitch almost periodic functions and measures \cite{LSS}.

When $\chi=1$ is the trivial character, the Fourier--Bohr coefficient $a_{1}^\cA(f)$ is usually called the \emph{mean of $f$}  and also denoted by $M_{\cA}(f)$.
That is, the mean $M_{\cA}(f)$ is given by
\[
M_{\cA}(f)=\lim_{n\to\infty} \frac{1}{|A_n|} \int_{A_n} f(s) \dd s \ts,
\]
if the limit exists.

%Furthermore, they exist uniformly for Weyl almost periodic functions/measures \cite{LSS}.

%\subsection{Besicovitch topology for measures}
%\marginpar{Do we need this section? D}

%Here, we recall some definitions and properties of \cite{LSS4}.

%\begin{definition}
%Let $\cA$ be a fixed van Hove sequence, and let $1 \leqslant p <\infty$. We define
%\[
%\cM_{b,p, \cA}:= \{ \mu \in \cM(G)\ts :\ts \| \mu*\varphi \|_{b, p, \cA} < \infty \text{ for all } \varphi \in \Cc(G) \} \ts.
%\]
%Each $\varphi \in \Cc(G)$ defines a semi-norm on $\cM_{b,p,\cA} (G)$ via
%\[
%N_{\varphi, p, \cA}(\mu):= \| \varphi*\mu \|_{b,p ,\cA} \ts.
%\]
%For a function $f\in \Cu(G)$, we define $N_{\varphi, p, \cA}(f):=N_{\varphi, p, \cA}(f\theta_G)$.
%\end{definition}

%It is not difficult to verify that any translation bounded measure is an element of $\cM_{b,p,\cA}(G)$, for all $1 \leqslant p < \infty$.

%Next, we define the equivalence relation on $\cM_{b,p,\cA}(G)$ via
%\[
%\mu \equiv_{b,p,\cA} \nu \iff N_{\varphi, p, \cA} (\mu-\nu)=0 \quad \text{ for all } \varphi \in \Cc(G) \ts.
%\]

%\begin{definition}\cite{LSS4} The locally convex topology on $\cM_{b,p,\cA}$ given by the family of semi-norms $\{ N_{\varphi,p, \cA} \ts :\ts \varphi \in \Cc(G)\}$  is called the \textit{Besicovitch topology for measures}. We will denote this topology by $\tau_{b,p,\cA}$.
%\end{definition}

%The Besicovitch topology is Hausdorff on $\cM_{b,p,\cA}/\equiv$ but not on $\cM_{b,p,\cA}$.

%\medskip

\section{The (twisted) Eberlein convolution for functions}
In this section, we review the Eberlein convolution and define the
twisted Eberlein convolution of functions and study their
properties. While the twisted Eberlein convolution differs from the
Eberlein convolution only slightly, it will be shown to behave
better in various respects.

\medskip

Let us start by recalling the known definition of the Eberlein
convolution.  Let  $f, g :G \to \CC$  be measurable functions and
$\cA$ a van Hove sequence. If the limit
\[
(f\ebeA g)(t):= \lim_{n\to\infty} \frac{1}{|A_n|} \int_{A_n}
f(s)\ts g(t-s) \ \dd s
\]
exists for all $t \in G$, we call the function  $f \ebeA g$ on $G$
the \textit{Eberlein convolution} of $f$ and $g$ and say that the
Eberlein convolution of $f$ and $g$ exists.

The twisted Eberlein convolution will be defined by a slight variant
of this definition.

\begin{definition}[Twisted Eberlein convolution]
Let measurable $f, g :G \to \CC$  and a van Hove sequence $\cA$ on
$G$ be given. If the limit
\[
\leftbrac f , g \rightbrac_{\cA}(t):= \lim_{n\to\infty}
\frac{1}{|A_n|} \int_{A_n} f(s)\ts\overline{g(s-t)} \ \dd s =
M_{\cA}(f \overline{\tau_t g})
\]
exists for all $t \in G$ we call the function $\leftbrac f , g
\rightbrac_{\cA}$ on $G$
 the \textit{twisted Eberlein convolution} of $f$ and $g$ and say that the twisted Eberlein convolution of $f$ and $g$ exists.
\end{definition}

\begin{remark}[Relating Eberlein convolution and twisted Eberlein
convolution]
The twisted Eberlein convolution $\leftbrac f , g \rightbrac_{\cA}$
exists if and only if $f \ebeA \widetilde{g}$ exists. Moreover, in
this case $\leftbrac f , g \rightbrac_{\cA}= f \ebeA \widetilde{g}$.  \exend
\end{remark}

We now look at properties of the (twisted) Eberlein convolution.
From the definition we immediately obtain the following connection
to the Besicovitch norm.

\begin{lemma}[Besicovitch norm via twisted Eberlein convolution]
Let $f\in L^\infty (G)$ be given and assume that $\leftbrac f , f
\rightbrac_{\cA}$ exists. Then,
\[
\leftbrac f , f \rightbrac_{\cA}(0) =  \| f \|_{b, 2, \cA}^2
\]
holds.  \qed
\end{lemma}

We now consider how the (twisted) Eberlein convolution behaves under
interchanging the arguments. We start by looking at the Eberlein
convolution.  For $f,g \in \Cu(G)$ the following result is proven in
\cite{BS}. Since the proof is identical in the general case, we skip
it.

\begin{lemma}\label{lem:ebe comm}
Let $f,g \in L^\infty(G)$ be such that $f\ebeA g$ is well defined. Then $g \circledast_{-\cA} f$ is well defined and
\[
g \circledast_{-\cA} f =f \ebeA g \ts.  \tag*{\qed}
\]
\end{lemma}

The corresponding result for the twisted version of the Eberlein
convolution shows the advantage of the twisted Eberlein convolution
as it does not require to replace the sequence $\cA$ by $-\cA$.

\begin{lemma}[Conjugate symmetry of twisted Eberlein convolution]
Let $f,g \in L^\infty(G)$ be such that $\leftbrac f , g
\rightbrac_{\cA}$ exists. Then, $\leftbrac g , f \rightbrac_{\cA}$
exists as well and
\begin{equation}\label{eq tebe twist com}
\leftbrac g , f \rightbrac_{\cA} = \widetilde{\leftbrac f , g
\rightbrac_{\cA}} \ts.
\end{equation}
\end{lemma}
\begin{proof}
Since $f \ebeA \widetilde{g}= \leftbrac f , g \rightbrac_{\cA}$ is
well defined, so is $\widetilde{g} \circledast_{-\cA} f$ by
Lemma~\ref{lem:ebe comm}. Therefore, the following limit exists for
all $t \in G$:
\begin{align*}
\overline{(f \ebeA \widetilde{g}) (-t)}
    &=\overline{\widetilde{g} \circledast_{-\cA} f (-t)}
     =  \lim_{n\to\infty} \frac{1}{|A_n|} \overline{ \int_{-A_n} \widetilde{g}(s)\ts f(-t-s) \ \dd s} \\
    &=  \lim_{n\to\infty} \frac{1}{|A_n|} \int_{-A_n} g(-s)\ts\overline{f(-t-s)} \ \dd s
     =  \lim_{n\to\infty} \frac{1}{|A_n|} \int_{A_n} g(u)\ts\overline{f(-t+u)} \ \dd u \\
    &=  \lim_{n\to\infty} \frac{1}{|A_n|}\int_{A_n} g(u)\ts\widetilde{f}(t-u) \ \dd u
     =  (g \ebeA \widetilde{f})(t)\ts.
\end{align*}
This  proves the claim.
\end{proof}

We now turn to a result on the twisted Eberlein convolution.
 First, let us recall that a function $f: G \to \CC$ is called \textit{positive definite}
if, for all $n \in \NN$, $x_1, \ldots, x_n \in G$ and $c_1,\ldots, c_n \in \CC$, we have
\[
\sum_{i,j=1}^n c_i f(x_i-x_j) \overline{c_j} \geqslant 0 \ts.
\]
For properties of positive definite functions we refer the reader to \cite{BF,MoSt}.

\begin{lemma}[Positive definiteness]
 Let $f\in L^\infty(G)$ be given. Then,  $\leftbrac f , f
\rightbrac_{\cA}$ is positive definite if it exists.
\end{lemma}
\begin{proof}
The proof is similar to the argument of \cite[Page 189]{MoSt}. Let  $n \in \NN, x_1, \ldots, x_n \in G$ and $c_1,\ldots, c_n \in \CC$.
Then,
\begin{align*}
\sum_{i,j=1}^n c_i\ts \leftbrac f , f\rightbrac_{\cA} (x_i-x_j)\ts \overline{c_j}
    &=\sum_{i,j=1}^n c_i\ts M_{\cA} (f \cdot \overline{\tau_{x_i-x_j} f})\ts \overline{c_j}
     =\sum_{i,j=1}^n c_i\ts M_{\cA} (\tau_{-x_i}f\cdot \overline{\tau_{-x_j} f})\ts \overline{c_j} \\
    &=M_{\cA} \Big(\sum_{i,j=1}^n c_i \tau_{-x_i}f\cdot  \overline{c_j \tau_{x_j} f}\Big)
     = M_{\cA} \Big( \Big|\sum_{i=1}^n c_i \tau_{-x_i}f \Big|^2\Big) \geqslant 0 \ts.
\end{align*}
\end{proof}

Next, let us look at how the (twisted) Eberlein convolution behaves
with respect to translation. We first look at the Eberlein
convolution.

\begin{lemma}
Let $f,g: G \to \CC$ be such that $f\ebeA g$ is well defined.
\begin{itemize}
\item[(a)] For all $t \in G$,  $f \ebeA (\tau_t g)$ is well defined and
\[
f \ebeA (\tau_{t} g) = \tau_{t} (f \ebeA g)  \ts.
\]
\item[(b)]  If $f,g \in L^\infty(G)$, then for all $t \in G$, $(\tau_t f) \ebeA g$ is well defined and
\[
(\tau_t f) \ebeA g = \tau_t (f \ebeA g)  \ts.
\]
\end{itemize}
\end{lemma}
\begin{proof}
(a) Since $f \ebeA g$ exists, the limits below exist for all $x,t$ and
\begin{align*}
\tau_t(f \ebeA g) (x)
    &= (f \ebeA g)(x-t)= \lim_{n\to\infty}  \frac{1}{|A_n|} \int_{A_n} f(s)\ts g(x-t-s)\ \dd s \\
    &= \lim_{n\to\infty}  \frac{1}{|A_n|} \int_{A_n} f(s)\ts (\tau_{t}g)(x-s)\ \dd s
     =  (f \ebeA (\tau_{t}g))(x) \ts.
\end{align*}
This shows that $f \ebeA (\tau_{t}g)$ exists and
\[
f \ebeA (\tau_{t}g)= \tau_t ( f \ebeA g) \ts.
\]

\smallskip

\noindent (b) For all $x,t \in G$ and all $n\in\NN$, we have
\begin{align*}
{}
    &\left|\frac{1}{|A_n|} \int_{A_n} f(s)\ts g(x-t-s)\ \dd s - \frac{1}{|A_n|} \int_{A_n}
       (\tau_t f)(s)\ts g(x-s)\ \dd s \right| \\
    &\phantom{========} =\frac{1}{|A_n|} \left|\int_{A_n} f(s)\ts  g(x-t-s) \ \dd s - \int_{A_n}
       f(s-t)\ts g(x-s)\ \dd s \right| \\
    &\phantom{========} =\frac{1}{|A_n|} \left|\int_{A_n} f(s)\ts  g(x-t-s) \ \dd s - \int_{-t+A_n}
       f(s)\ts g(x-t-s)\ \dd s \right| \\
    &\phantom{========} = \frac{1}{|A_n|} \left|\int_{A_n \Delta (-t+A_n)} f(s)\ts g(x-s-t)\ \dd s
      \right|  \ts.
\end{align*}

Since $(A_n)_{n\in\NN}$ is a F\o lner sequence and $f,g$ are bounded, we have
\[
\lim_{n\to\infty} \left|\frac{1}{|A_n|} \int_{A_n} f(s)\ts g(x-t-s)\ \dd s - \frac{1}{|A_n|} \int_{A_n}
       (\tau_t f)(s)\ts g(x-s)\ \dd s \right| =0 \ts.
\]
Therefore, since $\big(\frac{1}{|A_n|} \int_{A_n} f(s)\ts g(x-t-s)\ \dd s\big)_{n\in\NN}$ converges to $\tau_t ( f \ebeA g)$, it follows that
\[
\lim_{n\to\infty} \frac{1}{|A_n|} \int_{A_n}(\tau_t f)(s)\ts g(x-s)\ \dd s = \tau_t ( f \ebeA g)
\]
for all $x,t \in G$. This shows that $(\tau_tf)\ebeA g$ exists and
\[
(\tau_tf)\ebeA g=\tau_t ( f \ebeA g) \ts.   \qedhere
\]
\end{proof}

Similar results can be proven in the same way for the twisted
Eberlein convolution and this proves (a) and (b) of the next lemma.
Part (c) is then an immediate remarkable consequence of (a) and (b)
and shows a structural advantage of the twisted Eberlein
convolution.

\begin{lemma}[Translation invariance of twisted Eberlein convolution] \label{lem conv trans}
Let $f,g: G \to \CC$ be such that $\leftbrac f , g \rightbrac_{\cA}$
is well defined.
\begin{itemize}
\item[(a)] $\leftbrac f ,\tau_t g \rightbrac_{\cA}$ is well defined for all $t \in G$ and
\[
\leftbrac f ,\tau_t g \rightbrac_{\cA} = \tau_{-t} \leftbrac f , g
\rightbrac_{\cA}  \ts.
\]
\item[(b)]  If $f,g \in L^\infty(G)$, then  $\leftbrac \tau_t f , g \rightbrac_{\cA}$ is well defined for all $t \in G$ and
\[
\leftbrac \tau_t f , g \rightbrac_{\cA} = \tau_t \leftbrac f , g
\rightbrac_{\cA}  \ts.
\]
\item[(c)]
If $f,g \in L^\infty(G)$, then
\[
\leftbrac \tau_t f , \tau_t g
\rightbrac_{\cA} =\leftbrac f ,  g \rightbrac_{\cA}
\]
holds for all $t\in G$.  \qed
\end{itemize}
\end{lemma}

\begin{remark}
While at the first glance it may look like the Eberlein convolution
has a nicer behaviour with respect to translates it turns out that
the twisted version behaves better in the sense that it is
translation invariant.       \exend
\end{remark}

On a structural level our preceding  results give that  the twisted
Eberlein convolution can be seen as a form of translation invariant
inner product. Specifically, the following holds.

\begin{prop}[Inner-product like properties of twisted Eberlein convolution]
\label{prop:ip}
Choose $f,g \in L^\infty(G)$. If $\leftbrac f ,  g
\rightbrac_{\cA}$ exists, then
\[
\leftbrac \tau_t f , \tau_t g \rightbrac_{\cA} =\leftbrac f ,  g \rightbrac_{\cA} = \tau_t \leftbrac f ,  \tau_t g \rightbrac_{\cA}
\]
holds for all $t\in G$. Moreover, as long as the following twisted
 Eberlein convolutions exist they have the following properties:
 \begin{itemize}
\item $\leftbrac g , f \rightbrac_{\cA}=
\widetilde{\leftbrac f , g \rightbrac_{\cA}}$ \hspace*{\fill}
Conjugate symmetry
  \item $\leftbrac af+bh , g \rightbrac_{\cA}= a\leftbrac f , g \rightbrac_{\cA}+b\leftbrac h , g \rightbrac_{\cA}$ \hspace*{\fill} Linearity in the first argument
  \item $\leftbrac f , ag+bh \rightbrac_{\cA} = \bar{a}\leftbrac f , g \rightbrac_{\cA}+\bar{b}\leftbrac f , h \rightbrac_{\cA}$  \hspace*{\fill} Sesqui-linearity in the second argument
  \item $\leftbrac f , f \rightbrac_{\cA}$ is positive definite \hspace*{\fill} Positive-definiteness
  \item $ \| \leftbrac f , g \rightbrac_{\cA}\|_\infty \leqslant \| f \|_{b, 2, \cA} \cdot  \| g \|_{b, 2,\cA}$ \hspace*{\fill} Cauchy--Schwarz type inequality
  \item $\leftbrac f , f \rightbrac_{\cA}(0)=  \| \leftbrac f , f \rightbrac_{\cA} \|_\infty = \| f \|_{b, 2,
\cA}^2$.
\end{itemize}
\end{prop}
\begin{proof} Translation invariance has already been shown.
The first and the fourth  bullet have already been shown. The second
and third bullet are easy to derive.

It remains to show Cauchy-Schwarz inequality: By positive
definitedness we have
\[
\|\leftbrac f , f \rightbrac_{\cA} \|_\infty=
\leftbrac f , f \rightbrac_{\cA}(0)
\]
and as shown above we have
\[
\leftbrac f , f \rightbrac_{\cA}(0)
 =  \| f \|_{b, 2, \cA}^2 \ts.
\]
Now, consider $f,g$ such that the twisted Eberlein convolution
exists on the  linear span of their translates. Then, the map
\[
(u,v) \mapsto\leftbrac u , v
\rightbrac_{\cA}(0)=:\langle u, v\rangle
\]
is an (semi-)inner product
on this space by what we have shown already. Hence, it satisfies
\[
|\langle u, v\rangle|\leqslant \langle u, u\rangle^{1/2}  \langle
v,v\rangle^{1/2} \ts.
\]
Taking $u = f$, $v = \tau_t g$ we find from
translation invariance
\[
|\leftbrac f , g \rightbrac_{\cA}(t)| =|\tau_{-t} \leftbrac f , g
\rightbrac_{\cA}(0)| = |\leftbrac f , \tau_t g \rightbrac_{\cA}(0)|
=|\langle u,v\rangle| \ts.
\]
Putting this together we infer the statement.
\end{proof}

\begin{remark}
The preceding proposition shows the advantage of the
twisted Eberlein convolution over the Eberlein convolution.  \exend
\end{remark}

\subsection{Continuity and smoothing of the (twisted) Eberlein
convolution}
In this section we study continuity properties of the
(twisted) Eberlein convolution in its two arguments. As a
consequence we will be able to derive smoothing properties of the
(twisted) Eberlein convolution.

\medskip

The following result is an immediate consequence of \cite[Lemma~1.15]{LSS} and \cite[Lemma~1.13]{LSS} (note that the case $p=1$ trivially holds).

\begin{lemma}\label{lemma norm}
Let $p\geqslant1$ and let $q$ be such that
$\frac{1}{p}+\frac{1}{q}=1$ . Let  $f \in
L^p_{\operatorname{loc}}(G)$ and $g \in L^q_{\operatorname{loc}}(G)$
be such that $\leftbrac f , g \rightbrac_{\cA}$ exists. Then, we
have
\[
| \leftbrac f , g \rightbrac_{\cA}(t) | \leqslant \| f \|_{b, p,
\cA} \cdot \| \tau_t g \|_{b, q, \cA} \ts,
\]
for all $t \in G$. In particular, if $g \in L^\infty(G)$, we have
\[
\| \leftbrac f , g \rightbrac_{\cA} \|_\infty \leqslant \| f \|_{b,
p, \cA} \cdot  \| g \|_{b, q, \cA} \ts.   \tag*{\qed}
\]
\end{lemma}

\begin{coro}
Let $p\geqslant1$ and let $q$ be such that  $\frac{1}{p}+\frac{1}{q}=1$ . Let  $f \in L^p_{\operatorname{loc}}(G)$ and $g \in L^q_{\operatorname{loc}}(G)$ be such that $f \ebeA g$ exists. Then, for all $t \in G$, we have
\[
| (f \ebeA g)(t) | \leqslant \| f \|_{b, p, \cA} \cdot \| \tau_{-t} g \|_{b, q, -\cA} \ts.
\]
In particular, if $g \in L^\infty(G)$, we have
\[
\| f \ebeA g \|_\infty \leqslant \| f \|_{b, p, \cA} \cdot  \| g \|_{b, q, -\cA} \ts.    \tag*{\qed}
\]
\end{coro}

Let us note here that the previous twin results emphasize why it is sometimes advantageous to work with the twisted version of the Eberlein convolution.
A nice immediate consequence of Lemma~\ref{lemma norm} is the following result which states that (for suitable functions) the twisted Eberlein convolution is uniformly continuous.

\begin{coro}\label{core tebe cu}
Let $f \in BL_{\cA}^p(G)$ and $g \in BC_{\cA}^q(G)\cap L^\infty(G)$
with $\frac{1}{p}+\frac{1}{q}=1$ be such that $\leftbrac f , g
\rightbrac_{\cA}$ exists. Then,
\[
\leftbrac f , g \rightbrac_{\cA} \in \Cu(G) \ts.
\]
\end{coro}
\begin{proof}
First note that the function $\leftbrac f , g \rightbrac_{\cA}$ is
bounded by Lemma~\ref{lemma norm}. Next, by Lemma~\ref{lem conv
trans} and Lemma~\ref{lemma norm}, we have
\[
\| \tau_t \leftbrac f , g \rightbrac_{\cA}- \leftbrac f , g
\rightbrac_{\cA} \|_\infty  = \| \leftbrac f , \tau_{-t}g- g
\rightbrac_{\cA} \ \|_\infty
  \leqslant \| f \|_{b, p, \cA} \cdot  \| \tau_{-t}g- g \|_{b, q, \cA}
\]
for all $t \in G$.

Let $\eps >0$ be fixed but arbitrary.
Now, recall from \cite[p. 19]{LSS} that the translation mapping $\tau_t : \Cu(G) \to \Cu(G)$ has a unique extension to a continuous isometry $T_t : BC_{\cA}^q(G)/\equiv \to  BC_{\cA}^q(G)/\equiv$. Therefore, there exists an open neighbourhood $U =-U$ of $0$ such that
\[
\| T_t[g]- [g] \|_{b, q, \cA}< \frac{\eps}{ 1+\| f \|_{b, p, \cA} }
\]
for all $t \in U$.
Next, since $g \in L^\infty(G)$, \cite[Prop.~1.20]{LSS} implies
\[
T_t [g] = [\tau_t g] \ts.
\]
Therefore, we have $-t \in U$ for all $t \in U$, and hence
\[
\| \tau_{-t}g- g \|_{b, q, \cA}=\| T_{-t}[g]- [g] \|_{b, q, \cA}< \frac{\eps}{ 1+\| f \|_{b, p, \cA} } \ts.
\]
This gives
\[
\| \tau_t \leftbrac f , g \rightbrac_{\cA}- \leftbrac f , g
\rightbrac_{\cA} \|_\infty < \eps
\]
 for all $t \in U$, which proves the claim.
\end{proof}

It is possible to interchange the roles of $f$ and $g$ in the previous corollary due to Eq.~\eqref{eq tebe twist com}.

\begin{coro} \label{core tebe cu2}
Let $f \in BC_{\cA}^p(G) \cap L^\infty(G)$ and $g \in BL_{\cA}^q(G)$
with $\frac{1}{p}+\frac{1}{q}=1$ be such that $\leftbrac f , g
\rightbrac_{\cA}$ exists. Then,
\[
\leftbrac f , g \rightbrac_{\cA} \in \Cu(G) \ts.   \tag*{\qed}
\]
\end{coro}

We can now show that in many situations the twisted Eberlein convolution is a weakly almost periodic function.
Using properties of translation bounded measures, we will show latter in Corollary~\ref{cor1} that assuming the existence of
 $\leftbrac f , f \rightbrac_{\cA}$ and $\leftbrac g , g \rightbrac_{\cA}$ is actually not necessary.

\begin{coro}\label{cor2}
Let $f, g \in \Cu(G)$ be given and assume that $\leftbrac f , f
\rightbrac_{\cA}$, $\leftbrac g , g \rightbrac_{\cA}$ and $\leftbrac
f , g \rightbrac_{\cA}$ exist. Then, $\leftbrac f , g
\rightbrac_{\cA}$ is weakly almost periodic.
\end{coro}
\begin{proof} Any positive definite continuous function is weakly almost
periodic. Now, by polarisation, we can write $\leftbrac f , g
\rightbrac_{\cA}$  as linear combination of continuous positive definite
functions and the statement follows.
\end{proof}

The statements of Corollary~\ref{core tebe cu} and Corollary~\ref{core tebe cu2} remain true (except for changing $\cA$ to $-\cA$ at one point) if we consider the standard Eberlein convolution.

\begin{coro}
Let $p \geqslant 1$ and $q$ be such that $\frac{1}{p}+\frac{1}{q}=1$.
\begin{itemize}
  \item[(a)] Let $f \in BL_{\cA}^p(G)$ and $g \in BC_{-\cA}^q(G)\cap L^\infty(G)$ be such that $f \ebeA g$ exists. Then,
\[
f \ebeA g \in \Cu(G) \ts.
\]
  \item[(b)] Let $f \in BC_{\cA}^p(G) \cap L^\infty(G)$ and $g \in BL_{-\cA}^q(G)$ be such that $f \ebeA g$ exists. Then,
\[
f \ebeA g \in \Cu(G) \ts.  \tag*{\qed}
\]
\end{itemize}
\end{coro}

It turns out that the (twisted) Eberlein convolution is not only uniformly continuous and bounded, but it is also Bohr almost periodic if one of the two functions is mean almost periodic (see next theorem). This seems to be part of a more general smoothing effect the (twisted) Eberlein convolution has.
%Indeed, the (twisted) Eberlein convolution of weakly almost periodic functions/measures or more generally Besicovitch almost periodic functions/measures is strongly almost periodic. The phenomena is not that different than the well known smoothing effect the standard convolution has.

\begin{theorem}\label{them tebe is sap}
Let $f \in BL^{p}_{\cA}(G)\cap L^\infty(G)$ and  $g \in
Map^q_{\cA}(G) \cap L^\infty(G)$ with $\frac{1}{p}+\frac{1}{q}=1$.
If $\leftbrac f , g \rightbrac_{\cA}$ exists, then $\leftbrac f , g
\rightbrac_{\cA} \in SAP(G)$.
\end{theorem}
\begin{proof}
By Corollary~\ref{core tebe cu}, we have $\leftbrac f , g
\rightbrac_{\cA} \in \Cu(G)$ and hence it is continuous. Now, let
$\eps >0$. Since $g \in Map^p(G)$, the set
\[
P := \Big\{ t \in G\ts :\ts \| T_t g - g \|_{b,q,\cA} < \frac{\eps}{\| f\|_{b,p \cA}+1} \Big\}
\]
is relatively dense.
By Lemma~\ref{lemma norm}, we have
\[
\| T_t \leftbrac f , g \rightbrac_{\cA}- \leftbrac f , g
\rightbrac_{\cA} \|_\infty
    =  \| \leftbrac f , T_{-t}g-g \rightbrac_{\cA}  \|_\infty
    \leqslant \| f \|_{b, p, \cA} \| T_{-t}g-g \|_{b, q, \cA} < \eps
\]
for all $t \in -P$ and all $x \in G$. Since $P$ is relatively dense,
so is $-P$. Thus, $ \leftbrac f , g \rightbrac_{\cA} \in SAP(G)$, as
claimed.
\end{proof}

Once again, we can apply Eq.~\eqref{eq tebe twist com} to see that we can interchange the roles of $f$ and $g$.

\begin{coro}
Let $f \in Map^{p}_{\cA}(G)\cap L^\infty(G)$ and  $g \in
BL^q_{\cA}(G) \cap L^\infty(G)$ with $\frac{1}{p}+\frac{1}{q}=1$. If
$\leftbrac f , g \rightbrac_{\cA}$ exists, then $\leftbrac f , g
\rightbrac_{\cA} \in SAP(G)$.  \qed
\end{coro}

As before, the statements remain true for the standard Eberlein convolution.

\begin{coro}
Let $1 \leqslant p < \infty$ and let $q$ be such that $\frac{1}{p}+\frac{1}{q}=1$.
\begin{itemize}
  \item[(a)]Let $f \in BL^{p}_{\cA}(G)\cap L^\infty(G)$ and $g \in Map^q_{-\cA}(G) \cap L^\infty(G)$. If $f \ebeA g$ exists, then
$f \ebeA g \in SAP(G)$.
  \item[(b)]Let $f \in Map^{p}_{\cA}(G)\cap L^\infty(G)$ and $g \in BL^q_{-\cA}(G) \cap L^\infty(G)$. If $f \ebeA g$ exists, then
$f \ebeA g \in SAP(G)$.  \qed
\end{itemize}
\end{coro}

Let us complete this section by looking at the space of functions which are convolvable with a fixed function.

\begin{definition}
For $f : G \to \CC$ and a F\o lner sequence $(A_n)_{n\in\NN}$, we define
\begin{align*}
{\mathcal EB}_{\cA}(f)&:= \{ g : G \to \CC\ts :\ts  f \ebeA g \mbox{ is well defined} \} \ts, \\
{\mathcal TEB}_{\cA}(f)&:= \{ g : G \to \CC\ts :\ts  \leftbrac f , g
\rightbrac_{\cA} \mbox{ is well defined} \} \ts.
\end{align*}
\end{definition}

The following is an immediate consequence of Lemma~\ref{lemma norm} and Corollary~\ref{core tebe cu}.

\begin{lemma}
Let $1 \leqslant p < \infty$ and let $q$ be such that $\frac{1}{p}+\frac{1}{q}=1$.
\begin{itemize}
\item [(a)]Let $f \in BC^p_{\cA}(G)$. Then, the mapping
\[
(BL^q_{\cA}(G)\cap L^\infty(G)\cap {\mathcal TEB}_{\cA}(f) , \|
\cdot \|_{b,q,\cA}) \to (\Cu(G), \| \cdot \|_\infty) \ts, \qquad  g
\mapsto \leftbrac f , g \rightbrac_{\cA} \ts,
\]
is continuous.
\item [(b)]Let $f \in BL^p_{\cA}(G)$. Then, the mapping
\[
(BC^q_{\cA}(G)\cap L^\infty(G)\cap {\mathcal TEB}_{\cA}(f) , \|
\cdot \|_{b,q,\cA}) \to (\Cu(G), \| \cdot \|_\infty) \ts, \qquad g
\mapsto \leftbrac f , g \rightbrac_{\cA} \ts,
\]
is continuous.
\item [(c)] Let $f \in BC^p_{\cA}(G)$. Then, the mapping
\[
(BL^q_{-\cA}(G)\cap L^\infty(G)\cap {\mathcal EB}_{\cA}(f) , \| \cdot \|_{b,q,\cA}) \to (\Cu(G), \| \cdot \|_\infty) \ts, \qquad  g \mapsto f \ebeA g \ts,
\]
is continuous.
\item [(d)] Let $f \in BL^p_{\cA}(G)$. Then, the mapping
\[
(BC^q_{-\cA}(G)\cap L^\infty(G), \| \cdot \|_{b,q,\cA}) \to (\Cu(G), \| \cdot \|_\infty) \ts, \qquad g \mapsto f \ebeA g \ts,
\]
is continuous. \qed
\end{itemize}
\end{lemma}

\subsection{The twisted Eberlein convolution of Besicovitch almost periodic functions} \label{subsec:bes}
In the preceding sections we have been concerned with properties of
the (twisted) Eberlein convolution. There we have always assumed
that the twisted Eberlein convolution exists. Here, we  tackle the
issue of existence  and show that the twisted Eberlein convolution
of bounded Besicovitch almost periodic functions exists.

\medskip

We start with the following lemma, which was proven in a particular case in \cite{BS}. Since the proof for the general case is identical and straightforward, we skip it.

\begin{lemma}\label{bslemma} \cite[Lem. 4.1]{BS}
Let $\cA$ be a van Hove sequence, let $f\in L_{\text{loc}}^1(G)$, and let $\chi\in\widehat{G}$. The following statements are equivalent:
\begin{enumerate}
\item[(i)] $a^{\cA}_{\chi}(f)$ exists.
\item[(ii)] $\leftbrac f , \chi \rightbrac_{\cA}$ exists.
\item[(iii)] There is a $t\in G$ such that $\leftbrac f , \chi \rightbrac_{\cA}(t)$ exists.
\end{enumerate}
In that case, one has
\[
\leftbrac f , \chi \rightbrac_{\cA}(t) = \chi(t)\ts a^{\cA}_{\chi}(f)
\ts.        \tag*{\qed}
\]
\end{lemma}

\begin{remark}
By Lemma~\ref{bslemma}, for all $\chi \in \widehat{G}$ we have $\leftbrac \chi , \chi \rightbrac_{\cA}  = \chi$. Moreover, for all $\chi, \psi \in \widehat{G}$ with $\chi \neq \psi$ we have the orthogonality condition $\leftbrac \chi , \psi \rightbrac_{\cA} = 0$ .
\end{remark}

As an immediate consequence, we get the following result, compare \cite[Cor.~4.2]{BS}.

\begin{coro}
Let $f\in L_{\text{loc}}^1(G)$. Then, $a_{\chi}^{\cA}(f)$ exists for
all $\chi\in\widehat{G}$ if and only if $\leftbrac f , P
\rightbrac_{\cA}$ exists for all trigonometric polynomials
$P=\sum_{k=1}^nc_k\chi_k$.

In that case, one has
\[
\leftbrac f , P \rightbrac_{\cA} = \sum_{k=1}^n
a_{\chi_k}^{\cA}(f)\ts \overline{a_{\chi_k}^{\cA}(P)}\ts \chi_k \ts.  \tag*{\qed}
\]
\end{coro}

We can now prove the following result.

\begin{prop}\label{prop:tebe-bap-exists}
Let $f \in BL^2_{\cA}(G)$ be so that for all $\chi \in \widehat{G}$
the Fourier--Bohr coefficient $a_{\chi}^\cA(f)$ exists and let $g
\in Bap_{\cA}^2(G)\cap L^\infty(G)$. Then, $\leftbrac f , g
\rightbrac_{\cA}$ exists, $\leftbrac f , g \rightbrac_{\cA}\in
SAP(G)$ and
\[
a_{\chi}^\cA(\leftbrac f , g \rightbrac_{\cA})= a_{\chi}^\cA(f) \ts
\overline{a_{\chi}^\cA(g)}
\]
\end{prop}
\begin{proof}
By \cite[Lem. 3.5]{LSS}, there is a sequence $(P_n)_{n\in\NN}$ of trigonometric polynomials such that
\[
\|g-P_n\|_{b,2,\cA} \leqslant \frac{1}{n} \qquad \text{ and } \qquad \|P_n\|_{\infty} \leqslant \|g\|_{\infty} + 1 \ts.
\]
Also, \cite[Cor.~3.9]{LSS} implies $a_{\chi}^{\cA}(P_n)\xrightarrow{n\to\infty} a_{\chi}^{\cA}(g)$. Said corollary can be applied, since one has $\|g-P_n\|_{b,1,\cA} \leqslant \|g-P_n\|_{b,2,\cA}$.

Note that $g-P_n$ is bounded. Hence, we can apply Lemma~\ref{lemma norm} and get
\[
\| \leftbrac f , P_n-P_m \rightbrac_{\cA} \|_{\infty} \leqslant
\|f\|_{b,2,\cA} \ts \| P_n-P_m\|_{b,2,\cA} \ts.
\]
As
\[
\|P_n-P_m\|_{b,2,\cA} \leqslant \| P_n-g\|_{b,2,\cA} + \|g-P_m\|_{b,2,\cA} \leqslant \frac{1}{n} + \frac{1}{m} \ts,
\]
we find that $(\leftbrac f , P_n \rightbrac_{\cA})_{n\in\NN}$ is a
Cauchy sequence (of trigonometric polynomials, see previous
corollary) in $(SAP(G),\|\cdot\|_{\infty})$. Hence, it converges to
a function $h\in SAP(G)$.

Note that the first two assertions now follow from the following claim: For all $t\in G$, one has
\[
\lim_{m\to\infty} \frac{1}{|A_m|} \int_{A_m} f(s) \ts \overline{g(s-t)} \ \dd s= h(t) \ts.
\]
To prove the claim, let $t\in G$ and $\eps>0$. Pick $n\in \NN$ such that
\[
\| \leftbrac f , P_n \rightbrac_{\cA} - h \|_{\infty} <
\frac{\eps}{3} \qquad \text{ and } \qquad \|P_n-g\|_{b,2,\cA} <
\frac{\eps}{6\cdot(\|f\|_{b,2,\cA}+1)} \ts.
\]
Since $g$ is bounded, we also have $\|T_tP_n-T_tg\|_{b,2,\cA} < \frac{\eps}{6\cdot (\|f\|_{b,2,\cA}+1)}$. Therefore, there is $M_1\in \NN$ such that
\[
\frac{1}{|A_m|} \int_{A_m} \Big|\overline{P_n(s-t)} - \overline{g(s-t)} \Big|^2\ \dd s < \frac{\eps}{3\cdot(\|f\|_{b,2,\cA}+1)}
\]
for all $m>M_1$. At the same time, there is $M_2\in\NN$ such that
\[
\frac{1}{|A_m|} \int_{A_m} |f(s)|^2 \ \dd s < \|f\|_{b,2,\cA} + 1
\]
for all $m>M_2$. Together, the previous two inequalities and Cauchy--Schwarz' inequality yield
\[
\left| \frac{1}{|A_m|} \int_{A_m} f(s) \big( \overline{P_n(s-t)} - \overline{g(s-t)} \big) \ \dd s \right| \leqslant (\|f\|_{b,2,\cA}+1)\cdot  \frac{\eps}{3\cdot(\|f\|_{b,2,\cA}+1)} = \frac{\eps}{3}
\]
for all $m>M:=\max\{M_1,M_2\}$. Also, there is $M_3\in\NN$ such that
\[
\left| \frac{1}{|A_m|} \int_{A_m} f(s)\ts \overline{P_n(s-t)}\ \dd s
- \leftbrac f , P_n \rightbrac_{\cA}(t) \right| < \frac{\eps}{3}
\]
for all $m>M_3$. Putting everything together, we obtain
\begin{align*}
\left| \frac{1}{|A_m|} \int_{A_m} f(s)\ts \overline{g(s-t)}\
       \dd s-h(t) \right|
    &\leqslant \left| \frac{1}{|A_m|} \int_{A_m} f(s) \big(
      \overline{g(s-t)} - \overline{P_n(s-t)} \big)\dd s\right|   \\[2mm]
    &\phantom{XXXX}+ \left| \frac{1}{|A_m|} \int_{A_m} f(s)\ts
    \overline{P_n(s-t)} - \leftbrac f , P_n \rightbrac_{\cA}(t)
     \right|  \\[2mm]
    &\phantom{XXXX}+ |\leftbrac f, P_n \rightbrac_{\cA}(t)
      - h(t)| \\
    &<\frac{\eps}{3} + \frac{\eps}{3} + \frac{\eps}{3} = \eps \ts.
\end{align*}
Hence, the claim follows, and the first two assertions are proved.

Finally, since $\leftbrac f, P_n \rightbrac_{\cA}
\xrightarrow{n\to\infty} \leftbrac f, g \rightbrac_{\cA}$ with
respect to $\|\cdot\|_{\infty}$, we find
\[
a_{\chi}^{\cA}(\leftbrac f, g \rightbrac_{\cA})
    = \lim_{n\to\infty} \leftbrac f, P_n \rightbrac_{\cA}
     = \lim_{n\to\infty} a_{\chi}^{\cA}(f) \ts \overline{a_{\chi}^{\cA}(P_n)}
     = a_{\chi}^{\cA}(f)\ts \overline{a_{\chi}^{\cA}(g)} \ts.  \qedhere
\]
\end{proof}

Noting that in the proof of Proposition~\ref{prop:tebe-bap-exists} we only
used the boundedness of $g$ to deduce the translation invariance of
its Besicovitch norm, the following result is identical and we skip
its proof. For the definition and properties of Weyl almost periodic functions we refer the reader to \cite{LSS}.

\begin{prop}
Let $f \in BL^2_{\cA}(G)$ be so that for all $\chi \in \widehat{G}$
the Fourier--Bohr coefficient $a_{\chi}^\cA(f)$ exists and let $g$ be Weyl 2-almost periodic. Then, $\leftbrac f , g \rightbrac_{\cA}$
exists, $\leftbrac f , g \rightbrac_{\cA}\in SAP(G)$ and
\[
a_{\chi}^\cA(\leftbrac f , g \rightbrac_{\cA})= a_{\chi}^\cA(f) \ts
\overline{a_{\chi}^\cA(g)}  \,. \tag*{\qed}
\]
\end{prop}

Proposition~\ref{prop:tebe-bap-exists} has the following important consequence:
\begin{theorem}
Let $f,g  \in Bap_{\cA}^2(G)$ be so that $g \in L^\infty(G)$. Then,
$\leftbrac f , g \rightbrac_{\cA}$ exists, $\leftbrac f , g
\rightbrac_{\cA}\in SAP(G)$ and
\[
a_{\chi}^\cA(\leftbrac f , g \rightbrac_{\cA})= a_{\chi}^\cA(f) \ts
\overline{a_{\chi}^\cA(g)}\ts.  \tag*{\qed}
\]
\end{theorem}

With the help of the twisted Eberlein convolution a main result
of \cite{LSS} can be reformulated as follows (compare
\cite[Cor.~3.21]{LSS}).

\begin{theorem}
Let $f \in L^2_{\operatorname{loc}}(G) \cap L^\infty(G)$. Then, the following are equivalent:
\begin{itemize}
\item[(i)] $f \in Bap_{\cA}^2(G)$.
\item[(ii)]   \begin{itemize}
\item[(a)]$\leftbrac f , \chi \rightbrac_{\cA}$
exists for any character $\chi$.
\item[(b)] $\leftbrac f , f \rightbrac_{\cA}$ exists and $\leftbrac f , f \rightbrac_{\cA} \in SAP(G)$.
\item[(c)] For any character $\chi$, the equality
\[
a_{\chi}^\cA(\leftbrac f , f \rightbrac_{\cA})=|\leftbrac f , \chi \rightbrac_{\cA}|^2
\]
holds.
\item[(d)]  The identity
\[
\leftbrac f , f \rightbrac_{\cA} (0) = \sum_{\chi} |\leftbrac f , \chi \rightbrac_{\cA} |^2
\]
holds.  \qed
\end{itemize}
\end{itemize}
\end{theorem}

Combining  results in this section we get:

\begin{coro}
The mapping
\[
\leftbrac \cdot , \cdot  \rightbrac_{\cA} : (Bap_{\cA}^2(G) \cap
L^\infty(G)) \times (Bap_{\cA}^2(G) \cap L^\infty(G)) \to SAP(G)
\subseteq (Bap_{\cA}^2(G) \cap L^\infty(G))
\]
is well defined and satisfies the properties from Proposition~\ref{prop:ip}. \qed
\end{coro}

In fact, there is a converse to this corollary as follows.

\begin{theorem}
Let $\langle \cdot,\cdot \rangle : \bap_{\cA}^2(G) \cap L^\infty(G) \times
\bap_{\cA}^2(G) \cap L^\infty(G) \to \bap_{\cA}^2(G)\cap L^\infty(G)$ be any
function satisfying the statements of Proposition~\ref{prop:ip}. Then, $\langle
\cdot, \cdot \rangle$ is the twisted Eberlein convolution.
\end{theorem}
\begin{proof}

Note first that for $\chi \neq \psi \in \widehat{G}$ we have
\[
\langle \chi, \psi \rangle = \langle \tau_t \chi, \tau_t \psi \rangle
= (\chi-\psi)(t) \langle \chi, \psi \rangle
\]
showing that
\[
\langle \chi, \psi \rangle = 0 \qquad \text{ for all } \chi \neq \psi \in \widehat{G} \,.
\]
Next, for all $\chi \in \widehat{G}$ and $t \in G$ we have
\[
\chi(t) \langle \chi,  \chi \rangle = \langle \chi,  \tau_{-t} \chi \rangle = \tau_t \langle \chi, \chi \rangle
\]
showing that
\[
\langle \chi,  \chi \rangle(t)  =\chi(t)  \langle \chi,  \chi \rangle (0)= \chi(t) \|\chi \|_{b,2,\cA}^2 = \chi(t) \,.
\]
Therefore, for $\chi, \psi \in \widehat{G}$ we have
\[
\langle \chi, \psi \rangle =\delta_{\chi, \psi}\chi = \leftbrac \chi , \psi  \rightbrac_{\cA} \,.
\]

This shows that $\langle \cdot,\cdot \rangle $ agrees with the twisted
Eberlein convolution on $T \times T$ where $T$ is the set of
trigonometric polynomial. The Cauchy--Schwarz type inequality for
both $ \langle \cdot,\cdot \rangle $ and the twisted Eberlein convolution
gives us the continuity of these two functions from Besicovitch to
supremum norm topology. Therefore, by continuity and density, they agree.
\end{proof}

\section{The (twisted) Eberlein convolution for measures}
As mentioned in the introduction, quasicrystals are modeled using measures, not functions.
The notions of (twisted) Eberlein convolution carry naturally to measures.

\begin{definition}
We say that $\mu, \nu$ have a well defined \textit{Eberlein convolution} with respect to the van Hove sequence $\cA$ if
\[
\mu \ebeA \nu := \lim_{n\to\infty} \frac{1}{|A_n|}  \left( \mu|_{A_n} \right) * \left( \nu|_{-A_n} \right)
\]
exists in the vague topology. In this case, we call $\mu \ebeA \nu$ the \textit{Eberlein convolution} of $\mu$ and $\nu$.

We say that $\mu, \nu$ have a well defined \textit{twisted Eberlein convolution} with respect to the van Hove sequence $\cA$ if
\[
\leftbrac \mu , \nu \rightbrac_{\cA} := \mu \ebeA \widetilde{\nu}
\]
exists. In this case, we call $ \leftbrac  \mu, \nu \rightbrac_{\cA}
$ the \textit{twisted Eberlein convolution} of $\mu$ and $\nu$.

\end{definition}

Let us note here in passing that
\[
\leftbrac \mu , \nu \rightbrac_{\cA} =  \lim_{n\to\infty} \frac{1}{|A_n|}  \left( \mu|_{A_n} \right) * \widetilde{\left( \nu|_{A_n} \right)} \,,
\]
which makes the twisted version of the Eberlein convolution more natural, as we restrict both measures to $A_n$.

\smallskip

\noindent\fbox{%
    \parbox{15.3cm}{%
For the rest of the paper,  we assume that the (twisted) Eberlein
convolution exists, whenever we write $\mu \ebeA \nu$ or $ \leftbrac
\mu , \nu \rightbrac_{\cA} $, respectively.
        }%
}

\medskip

The following result shows the compatibility of the (twisted) Eberlein convolution for measures and functions.

\begin{prop}\label{prop1} \cite{LSS} Let $f,g \in \Cu(G), \mu,\nu \in \cM^\infty(G)$ and let $\cA$ be a van Hove sequence.
\begin{itemize}
\item[(a)] $f \ebeA g$ exists if and only if $(f \theta_G) \ebeA (g \theta_G)$ exists. Moreover, in this case
\[
(f \theta_G) \ebeA (g \theta_G)=(f \ebeA g) \theta_G \ts.
\]
\item[(b)] $\leftbrac f , g \rightbrac_{\cA}$ exists if and only if $ \leftbrac f \theta_G , g \theta_G \rightbrac_{\cA} $ exists. Moreover, in this case
\[
 \leftbrac f \theta_G ,  g \theta_G\rightbrac_{\cA} =\leftbrac f , g \rightbrac_{\cA} \theta_G \ts.
\]
\item[(c)] $ \leftbrac  \mu, \nu\rightbrac_{\cA}$ exists if and only if for all $\varphi, \psi \in \Cc(G)$ the limit
\[
M_{\cA}((\mu*\varphi)\cdot( \overline{ \nu*\psi }))=\lim_{n\to\infty} \frac{1}{|A_n|} \int_{A_n} (\mu*\varphi)(t)\ts \overline{ (\nu*\psi)(t)}\ \dd t
\]
exists. Moreover, in this case
\[
M_{\cA}((\mu*\varphi)\cdot( \overline{ \nu*\psi }))=\big(( \leftbrac
\mu ,  \nu\rightbrac_{\cA} )*(\varphi*\widetilde{\psi})\big)(0) \ts.
\]
\item[(d)] $\mu \ebeA \nu$ exists if and only if for all $\varphi, \psi \in \Cc(G)$ the limit
\[
M_{\cA}((\mu*\varphi)\cdot( \nu*\psi)^\dagger)=\lim_{n\to\infty} \frac{1}{|A_n|} \int_{A_n} (\mu*\varphi)(t)\ts ( \nu*\psi)(-t)\ \dd t
\]
exists. Moreover, in this case
\[
M_{\cA}((\mu*\varphi)\cdot( \nu*\psi)^\dagger)=\big((\mu \ebeA \nu)*(\varphi*\psi)\big)(0) \ts.
\]
\end{itemize}
\end{prop}
\begin{proof}
(a) is Proposition~1.5 in \cite{LSS}, (b) follows from (a), (c) is Proposition~1.4 in \cite{LSS}, and (d) follows from (c).
\end{proof}

As an immediate consequence we get:

\begin{coro}\label{lem87}
Let $\mu, \nu \in \cM^\infty(G)$ be such that $ \leftbrac  \mu, \nu
\rightbrac_{\cA}$ is well defined. Then, the twisted Eberlein
convolution $\leftbrac  \mu*\varphi, \nu*\psi \rightbrac_{\cA} $
exists for all $\varphi, \psi \in \Cc(G)$ and
\[
 \leftbrac  \mu*\varphi,  \nu*\psi\rightbrac_{\cA} =  \leftbrac \mu ,\nu \rightbrac_{\cA}  * \varphi *\widetilde{\psi} \ts.   \tag*{\qed}
\]
\end{coro}

 Similarly to functions, the twisted Eberlein convolution of measures satisfies a conjugate symmetry relation.
\begin{lemma}\label{lem tebe switch}
Let $\mu, \nu \in \cM^\infty(G)$ be such that $ \leftbrac \mu , \nu
\rightbrac_{\cA} $ is well defined. Then, $ \leftbrac \nu , \mu
\rightbrac_{\cA}$ is well defined and
\[
 \leftbrac  \nu,  \mu \rightbrac_{\cA} = \widetilde{ \leftbrac \mu, \nu \rightbrac_{\cA} } \ts.
\]
\end{lemma}
\begin{proof}
The claim follows immediately from the observation that
\[
  \widetilde{ \left( \mu|_{A_n} \right) * \widetilde{\left( \nu|_{A_n} \right)}} =  \widetilde{ \left( \mu|_{A_n} \right) }*\widetilde{ \widetilde{\left( \nu|_{A_n} \right)}}
  = \left( \nu|_{A_n} \right) *  \widetilde{ \left( \mu|_{A_n} \right) } \ts.  \qedhere
\]
\end{proof}

By switching between the Eberlein convolution and its twisted version, we get the following formula for the Eberlein convolution. Exactly as for functions, the Eberlein
convolution is commutative if the van Hove sequence is symmetric.

\begin{coro}
Let $\mu, \nu \in \cM^\infty(G)$ be such that $\mu \ebeA \nu$ is well defined. Then, $\nu \circledast_{-\cA} \mu$ is well defined and
\[
\nu \circledast_{-\cA} \mu = \mu \ebeA \nu \ts.  \tag*{\qed}
\]
\end{coro}

\medskip

Next, we give an alternate characterisation for the (twisted) Eberlein convolution. Note that this type of computations appeared implicitly in \cite{BL,LS,Martin2} just to name a few.

\begin{lemma}
Let $\mu , \nu \in \cM^\infty(G)$, and let $\cA$ be a van Hove sequence. Then, we have
\[
\lim_{n\to\infty} \frac{1}{|A_n|}\left( \mu- \mu|_{A_n} \right)* \widetilde{\left( \nu|_{A_n} \right)}=0
\]
in the vague topology.
\end{lemma}
\begin{proof}
Let $\varphi \in \Cc(G)$, and let $K= \supp(\varphi)$. Then,
\begin{align*}
\Big| \frac{1}{|A_n|} ( \mu- \mu|_{A_n} )*\widetilde{( \nu|_{A_n} )} \Big|
    & = \frac{1}{|A_n|}  \left| \int_{G} \int_{G} \varphi(s+t)\ \dd \widetilde{( \nu|_{A_n})}(s) \  \dd
      ( \mu- \mu|_{A_n} )(t) \right|\\
    &= \frac{1}{|A_n|}  \left| \int_{G} \overline{\int_{G} \overline{\varphi(-s+t)}\ts 1_{A_n}(s)\ \dd \nu (s)}
        \ts 1_{G \backslash A_n}(t)\ \dd \mu(t) \right| \\
    &= \frac{1}{|A_n|}  \left| \int_{G} \overline{\int_{G} \overline{\varphi(-s+t)}\ts 1_{A_n}(s)\ts
       1_{G \backslash A_n}(t)\ \dd \nu (s)}\ \dd \mu(t) \right|   \ts.
\end{align*}

Next, $\varphi(-s+t) 1_{A_n}(s)1_{G \backslash A_n}(t)\neq 0$ implies that $t-s \in K, s \in A_n, t \notin A_n$. Therefore, $t \in (A_n+K) \backslash A_n \subseteq \partial^k(A_n)$.
Hence, one has
\begin{align*}
\Big| \frac{1}{|A_n|} ( \mu- \mu|_{A_n})*\widetilde{( \nu|_{A_n} )} \Big|
    & = \frac{1}{|A_n|}  \left| \int_{G} \overline{\int_{G} \overline{\varphi(-s+t)}\ts 1_{A_n}(s)
       \ts 1_{\partial^K(A_n)}(t)\ \dd \nu (s)}\ \dd \mu(t) \right| \\
    &\leqslant \frac{1}{|A_n|} \int_{G} \int_{G} |\overline{\varphi(-s+t)}|\ts 1_{A_n}(s)\ts
       1_{\partial^K(A_n)}(t)\  \dd |\nu| (s)\ \dd |\mu|(t)  \\
    &= \frac{1}{|A_n|} \int_{G} 1_{\partial^K(A_n)}(t)  \int_{G} |\varphi(-s+t)|\ts  1_{A_n}(s)\ \dd |\nu| (s)
      \ \dd |\mu|(t)  \\
    &\leqslant  \frac{1}{|A_n|} \int_{G} 1_{\partial^K(A_n)}(t)  \int_{G} |\varphi(-s+t)|\ \dd |\nu| (s)\
       \dd |\mu|(t)  \\
   &= \frac{1}{|A_n|}  \int_{G} 1_{\partial^K(A_n)}(t)\ts (|\varphi| *|\nu|) (t) \  \dd |\mu|(t)  \\
   &\leqslant \| |\varphi| *|\nu| \|_\infty \cdot \frac{|\mu|(\partial^K(A_n))}{|A_n|}  \ts.
\end{align*}
Now, since $\nu$ is translation bounded, so is $|\nu|$, and hence $ \| \left| \varphi\right|*\left|\mu \right| \|_\infty$, see \cite{ARMA1,MoSt}. Moreover, one has $\lim_{n\to\infty} \frac{|\mu |(\partial^K(A_n))}{|A_n|} =0$ by \cite[Lemma~ 1.1]{Martin2}, which completes the proof.
\end{proof}

This result has the following interesting  consequences.

\begin{coro}\label{coro:second def Ebe}
Let $\mu, \nu \in \cM^\infty(G)$, and let $\cA= (A_n)_{n \in\NN}$ be a van Hove sequence. If one of the limits below exist, then all three exist and they are equal, i.e.
\[
 \leftbrac  \mu,  \nu \rightbrac_{\cA}  =\lim_{n\to\infty} \frac{1}{|A_n|}   \big(\mu|_{A_n}  * \widetilde{\nu|_{A_n}}\big)
=\lim_{n\to\infty} \frac{1}{|A_n|}  \big(\mu  * \widetilde{\nu|_{A_n}} \big)
=\lim_{n\to\infty} \frac{1}{|A_n|}  \big( \mu|_{A_n}  * \widetilde{\nu} \big)  \ts.  \tag*{\qed}
\]
\end{coro}

Similarly, we get the following relations for the Eberlein convolution.

\begin{coro}
Let $\mu , \nu \in \cM^\infty(G)$, and let $\cA$ be a van Hove sequence. Then,  we have
\[
\lim_{n\to\infty} \frac{1}{|A_n|}\left( \mu- \mu|_{A_n} \right)* \left( \nu|_{-A_n} \right)=0
\]
in the vague topology.  \qed
\end{coro}

\begin{coro}
Let $\mu, \nu \in \cM^\infty(G)$, and let $\cA= (A_n)_{n \in\NN}$ be a van Hove sequence. If one of the limits below exist, then all three exist and they are equal, i.e.
\[
 \mu \ebeA \nu = \lim_{n\to\infty} \frac{1}{|A_n|}   \big(\mu|_{A_n}  * \nu|_{-A_n}\big)
=\lim_{n\to\infty} \frac{1}{|A_n|}  \big(\mu  * \nu|_{-A_n} \big)
=\lim_{n\to\infty} \frac{1}{|A_n|}  \big( \mu|_{A_n}  * \nu \big)  \ts.  \tag*{\qed}
\]
\end{coro}

\medskip

Let us also briefly discuss the compatibility of the twisted Eberlein convolution with translations. Similarly to functions, we get the following formulas:

\begin{lemma}
Let $\mu,\nu\in \mc{M}^{\infty}(G)$, and let
$\mc{A}=(A_n)_{n\in\NN}$ be a van Hove sequence such that $
\leftbrac \mu , \nu \rightbrac_{\cA} $ exists. Then, the following
twisted Eberlein convolutions exist and
\[
\tau_t \leftbrac  \mu,  \nu \rightbrac_{\cA} = \leftbrac \tau_t\mu ,
\nu\rightbrac_{\cA}  = \leftbrac \mu , \tau_{-t} \nu
\rightbrac_{\cA}
\]
for all $t\in G$.
\end{lemma}
\begin{proof}
A simple computation yields
\[
\tau_t \left( \frac{1}{|A_n|} (\mu* \widetilde{\nu|_{A_n}})\right) = \frac{1}{|A_n|} ((\tau_t\mu)* \widetilde{\nu|_{A_n}}) = \frac{1}{|A_n|} (\mu* \widetilde{(\tau_{-t}\nu)|_{A_n}})  \ts.
\]
The claim follows.
\end{proof}

\begin{coro}
Let $\mu,\nu\in \mc{M}^{\infty}(G)$, and let $\mc{A}=(A_n)_{n\in\NN}$ be a van Hove sequence such that $\mu\ebeA \nu$ exists. Then,  the following Eberlein convolutions exist and
\[
\tau_t(\mu\ebeA \nu) = (\tau_t\mu)\ebeA \nu = \mu\ebeA (\tau_{t} \nu)
\]
for all $t\in G$.  \qed
\end{coro}

\begin{remark}
The twisted Eberlein convolution for measures satisfies the following translation compatibility relation: If $\mu, \nu \in \cM^\infty(G)$ and $\cA$ are such that
$\leftbrac \mu , \nu \rightbrac_{\cA} $ exists, then  the convolution $
\leftbrac \tau_t \mu , \tau_t \nu \rightbrac_{\cA} $ exists for all $t \in G$ and
\[
\leftbrac \tau_t \mu , \tau_t \nu \rightbrac_{\cA}=\leftbrac \mu , \nu \rightbrac_{\cA} \ts. \tag*{\exend}
\]
\end{remark}

Next we discuss the existence of the twisted Eberlein convolution. We start by proving the following preliminary result (compare \cite{BL}).

\begin{lemma}\label{lem:tebe exists subseq}
Let $\mu,\nu\in\mc{M}^{\infty}(G)$ and $\mc{A}=(A_n)_{n\in\NN}$ a van Hove sequence. Then, there exist a constant $C >0$ and a precompact open set $U$ such that
\[
\frac{1}{|A_n|}   ( \mu|_{A_n} * \widetilde{\nu}) \in \cM_{C,U} \ts.
\]
Moreover, $\cM_{C,U}$ is vaguely compact and metrisable.
\end{lemma}
\begin{proof}
Let $K,K'$ be compact sets, and let $U\neq \varnothing$ be precompact such that $U \subseteq K \subseteq (K')^{\circ}$.
First note that the sequence $\left(|A_n|^{-1}\cdot |\mu|(A_n)\right)_{n\in\NN}$ is bounded from above by a constant $C'>0$; see \cite[Lemma~1.1.(2)]{Martin2}.
Set $C:= C' \| \widetilde{\nu} \|_{K'}$.
Then, we have
\[
\Big|\frac{1}{|A_n|}  \big( \mu|_{A_n}  * \widetilde{\nu} \big) \Big|(t+K) \leqslant \| \widetilde{\nu}\|_{t+K'} \cdot |\mu|(A_n)  \leqslant C' \| \widetilde{\nu} \|_{K'} =C
\]
by \cite[Lemma~6.1]{NS11}.
Taking the supremum over all $t\in G$, we obtain
\[
\Big\|\frac{1}{|A_n|}  \big( \mu|_{A_n}  * \widetilde{\nu} \big)\Big\|_{K} \leqslant  C \ts.
\]
Since $U \subset K$ we thus get in particular
\[
\Big\|\frac{1}{|A_n|}  \big( \mu|_{A_n}  * \widetilde{\nu} \big)\Big\|_{U} \leqslant  C \ts,
\]
which proves the first claim.

Now, as $G$ is second countable, $\cM_{C,U}$ is metrisable and compact by \cite[Thm.~2]{BL}.
\end{proof}

As an immediate consequence of Lemma~\ref{lem:tebe exists subseq}, we get the following result which shows the existence of (twisted) Eberlein convolution of arbitrary
measures along subsequences, as well as the biliniarity of the (twisted) Eberlein convolution.

\begin{theorem}
Let $\mu,\nu,\sigma\in\mc{M}^{\infty}(G), a,b \in \CC$ and let $\cA$ be an arbitrary van Hove sequence. Then,
\begin{itemize}
  \item[(a)] there exists a subsequence $\cB$ of $\cA$ such that $ \leftbrac  \mu, \nu  \rightbrac_{\cB}$ exists.
  \item[(b)] there exists a subsequence $\cB$ of $\cA$ such that $\mu \ebeB \nu$ exists.
  \item[(c)] there exists a subsequence $\cB$ of $\cA$ such that $  \leftbrac  a\mu+b\nu, \sigma \rightbrac_{\cB} ,  \leftbrac \mu , \sigma \rightbrac_{\cB} ,  \leftbrac \nu , \sigma \rightbrac_{\cB} $ exist. Moreover, for all subsequences
  for which these twisted Eberlein convolutions exist we have
\[
 \leftbrac a\mu+b\nu , \sigma \rightbrac_{\cB} =a \leftbrac   \mu,  \sigma \rightbrac_{\cB} + b  \leftbrac \nu , \sigma \rightbrac_{\cB}  \ts.
\]
\item[(d)] There exists a subsequence $\cB$ of $\cA$ such that $ \leftbrac  \mu, a \nu+b\sigma \rightbrac_{\cB} ,  \leftbrac \mu , \nu \rightbrac_{\cB} , \leftbrac \mu, \sigma  \rightbrac_{\cB} $ exist. Moreover, for all subsequences
  for which these twisted Eberlein convolutions exist, we have
\[
 \leftbrac \mu , a \nu+b\sigma \rightbrac_{\cB} =\bar{a} \leftbrac  \mu , \nu \rightbrac_{\cB} + \bar{b}  \leftbrac \mu , \sigma  \rightbrac_{\cB} \ts.
\]
  \item[(e)] There exists a subsequence $\cB$ of $\cA$ such that $(a\mu+b\nu) \ebeB \sigma, \mu \ebeB \sigma, \nu \ebeB \sigma$ exist. Moreover, for all subsequences
  for which these Eberlein convolutions exist we have
\[
(a\mu+b\nu) \ebeB \sigma=a( \mu \ebeB \sigma)+b( \nu \ebeB \sigma) \ts.
\]
\item[(f)] There exists a subsequence $\cB$ of $\cA$ such that $\mu \ebeB (a \nu+b\sigma), \mu \ebeB \nu, \mu \ebeB \sigma$ exist. Moreover, for all subsequences
  for which these twisted Eberlein convolutions exist we have
\[
\mu \ebeB (a \nu+b\sigma)=a \mu \ebeB \nu+ b \mu \ebeB \sigma \ts.   \tag*{\qed}
\]
\end{itemize}
\end{theorem}

\medskip

Next, we want to move the Besicovitch topology to measures, compare \cite{LSS4}. In order to do this, note that the
Besicovitch semi-norm for functions induces a family $\{ N_{\varphi,p, \cA} : \varphi \in \Cc(G) \}$ of semi-norms on $\cM^\infty(G)$ via
\[
N_{\varphi, p, \cA}(\mu):= \| \varphi*\mu \|_{b,p ,\cA} \ts.
\]

By combining Lemma~\ref{lem87} with Lemma~\ref{lemma norm}, we immediately obtain the following estimates.

\begin{prop}\label{prop 2}
Let $\mu, \nu \in \cM^\infty(G)$ be such that $ \leftbrac \mu , \nu
\rightbrac_{\cA} $ is well defined and $\varphi, \psi \in \Cc(G)$.
Then,
\begin{itemize}
  \item[(a)] For all $1 < p,q <\infty$ conjugate, we have
\[
\| \leftbrac  \mu , \nu \rightbrac_{\cA}
\|_{\varphi*\widetilde{\psi}} \leqslant N_{\varphi, p , \cA}( \mu)
N_{\psi, q , \cA}(\nu) \ts.
\]
  \item[(b)] In particular, one has
\[
\| \leftbrac  \mu , \nu \rightbrac_{\cA}
\|_{\varphi*\widetilde{\psi}} \leqslant N_{\varphi, 2 , \cA}( \mu)
N_{\psi, 2 , \cA}(\nu) \ts.
\]
\item[(c)] Let $K =\supp(\psi)$. Then,
\[
\|  \leftbrac \mu ,  \nu\rightbrac_{\cA}  \|_{\varphi*\widetilde{\psi}} \leqslant N_{\varphi, 1 , \cA}( \mu) \| \nu*\psi \|_\infty
\leqslant N_{\varphi, 1 , \cA}( \mu) \|\psi \|_\infty  \| \nu \|_K \ts.  \tag*{\qed}
\]
\end{itemize}
\end{prop}

 As an immediate consequence we get similar estimates for the Eberlein convolution.

\begin{coro}
Let $\mu, \nu \in \cM^\infty(G)$ be such that $\mu \ebeA \nu$ is well defined and $\varphi, \psi \in \Cc(G)$.
\begin{itemize}
  \item[(a)] For all $1 < p,q <\infty$ conjugate, we have
\[
\| \mu \ebeA \nu \|_{\varphi*\widetilde{\psi}} \leqslant N_{\varphi, p , \cA}( \mu) N_{\psi, q , -\cA}(\nu) \ts.
\]
  \item[(b)] In particular, one has
\[
\| \mu \ebeA \nu \|_{\varphi*\widetilde{\psi}} \leqslant N_{\varphi, 2 , \cA}( \mu) N_{\psi, 2 , -\cA}(\nu) \ts.
\]
\item[(c)] Let $K =\supp(\psi)$. Then
\[
\| \mu \ebeA \nu \|_{\varphi*\widetilde{\psi}} \leqslant N_{\varphi, 1 , \cA}( \mu) \| \nu*\psi \|_\infty
\leqslant N_{\varphi, 1 , \cA}( \mu) \|\psi \|_\infty  \| \nu \|_K \ts.  \tag*{\qed}
\]
\end{itemize}
\end{coro}

Similarly with functions, let us introduce the following definition.

\begin{definition}
For $\mu \in \cM^\infty(G)$ we define
\begin{align*}
{\mathcal EB}_{\cA}(\mu)&:= \{ \nu \in \cM^\infty\ts :\ts  \mu \ebeA \nu \mbox{ is well defined} \}  \ts, \\
{\mathcal TEB}_{\cA}(\mu)&:= \{ \nu \in \cM^\infty\ts :\ts
\leftbrac\mu  ,  \nu\rightbrac_{\cA}  \mbox{ is well defined} \} \ts.
\end{align*}
\end{definition}

\smallskip
Now, on $\cM^\infty(G)$ consider the locally convex topology defined by $\{ N_{\varphi, 2, \cA} : \varphi \in \Cc(G) \}$, which we will call
the \emph{ Besicovitch 2-topology for measures}. Note here that a net $(\mu_\alpha)_{\alpha}$ in $\cM^\infty(G)$ converges in the Besicovitch 2-topology to some measure $\mu$ if and only if,
for all $\varphi \in \Cc(G)$, we have
\[
\lim_\alpha \| \mu_\alpha *\varphi -\mu *\varphi \|_{b,2,\cA} \,.
\]

The next result follows immediately from Proposition~\ref{prop 2}.

\begin{theorem}
Let $\mu\in\cM^{\infty}(G)$. The mapping
\[
{\mathcal TEB}_{\cA}(\mu) \to \cM^\infty(G),
 \,, \qquad \nu \mapsto \leftbrac \mu  ,
\nu \rightbrac_{\cA}  \,,
\]
is continuous from the Besicovitch 2-topology to the semi-product topology.  \qed
\end{theorem}

\begin{prop}
Let $\mu\in\cM^{\infty}(G)$. The mapping
\[
{\mathcal EB}_{\cA}(\mu) \to \cM^\infty(G) \,, \qquad \nu \mapsto \nu \ebeA \mu \,,
\]
is continuous from the Besicovitch 2-topology to the semi-product topology. \qed
\end{prop}

We can now show that if one of the measures is mean almost periodic, then the twisted Eberlein convolution is strong almost periodic.

\begin{theorem}\label{thm tebe is SAP}
Let $\mu, \nu \in \cM^\infty(G)$ be such that $ \leftbrac \mu , \nu
\rightbrac_{\cA} $ is well defined.
\begin{itemize}
  \item[(a)] If $\mu \in \cM\mathtt{ap}_{\cA}(G)$, then $ \leftbrac \mu , \nu \rightbrac_{\cA}  \in \SAP(G)$.
  \item[(b)] If $\nu \in \cM\mathtt{ap}_{\cA}(G)$, then $ \leftbrac \mu , \nu \rightbrac_{\cA}  \in \SAP(G)$.
\end{itemize}
\end{theorem}
\begin{proof}
(a) follows immediately from Lemma~\ref{lem87} and Theorem~\ref{them tebe is sap}. (b) follows from (a) and Lemma~\ref{lem tebe switch}.
\end{proof}

Setting $\mu=\nu$ in Theorem~\ref{thm tebe is SAP} leads to an
important result in Aperiodic Order. Here, one is interested in the
diffraction measure $\widehat{\gamma}$ of a given measure $\mu$,
i.e. the Fourier transform of the autocorrelation measure
$\gamma:=\leftbrac \mu , \mu \rightbrac_{\cA}$. The case of a pure
point diffraction measure is particularly important because they
describe highly ordered structures. However, $\widehat{\gamma}$ is a
pure point measure if and only if $\gamma$ is a strongly almost
periodic measure \cite[Cor.~4.10.13]{MoSt}. Thus, the following
corollary of Theorem~\ref{thm tebe is SAP} is relevant in this
context.

\begin{coro} \cite{LSS}
Let $\mu \in \cM^\infty(G)$ and let $\gamma$ be the autocorrelation of $\mu$ with respect to $\cA$. If
$\mu \in \cM\mathtt{ap}_{\cA}(G)$, then $\gamma \in \SAP(G)$.  \qed
\end{coro}

One can even show that $\mu \in \cM\mathtt{ap}_{\cA}(G)$ if and only
if $\gamma \in \SAP(G)$ \cite[Thm. 2.13]{LSS}. However, if $\mu$ and
$\nu$ are different measures, then $ \leftbrac \mu ,
\nu\rightbrac_{\cA}\in\SAP(G)$ does not imply that one of the
measures is mean almost periodic, as the following example shows.

\begin{example}
Let us consider the Bernoulli comb
\[
\omega_B = \sum_{m\in\ZZ} W_m\ts \delta_m \ts,
\]
where $(W_m)_{m\in\ZZ}$ is a family of i.i.d random variables (the
latter taking the values $1$ and $0$ with probabilities $p$ and
$1-p$ with $0<p<1$). Similar to \cite[Prop.~11.1]{TAO}, one can show
that its autocorrelation $ \leftbrac  \omega_B, \omega_B
\rightbrac_{\cA}$ is given by $p^2\delta_{\ZZ}+p(1-p)\delta_0$,
which is not a strongly almost periodic measure. By a previous
argument, $\omega_B$ cannot be mean almost periodic.

Next, let us consider two different generic Bernoulli combs
$\omega_B$ and $\omega_{B'}$. This time, a similar computation for the autocorrelation
gives $ \leftbrac \omega_B , \omega_{B'}
\rightbrac_{\cA}=p^2\delta_{\ZZ}$, which is a strongly almost
periodic measure. Therefore, $ \leftbrac  \omega_B, \omega_{B'}
\rightbrac_{\cA}$ is strongly almost periodic, although neither
$\omega_B$ nor $\omega_{B'}$ is mean almost periodic.  \exend
\end{example}

\begin{remark}
Let $\mu, \nu \in \cM^\infty(G)$ be such that $\mu \in
\cM\mathtt{ap}_{\cA}(G)$. If $ \leftbrac \mu , \nu \rightbrac_{\cA}$
does not exists, the sequence
\[
\frac{1}{|A_n|}   \mu|_{A_n}  * \widetilde{ \nu|_{A_n} }
\]
will have multiple cluster points by Lemma~\ref{lem:tebe exists subseq}. By applying Theorem~\ref{thm tebe is SAP} to the corresponding subsequences of $\cA$, we find that all these cluster points belong to $\SAP(G)$.
The same observation holds if $\mu \in \cM\mathtt{ap}_{\cA}(G)$ is replaced by $\nu \in \cM\mathtt{ap}_{\cA}(G)$. \exend
\end{remark}

Theorem~\ref{thm tebe is SAP} can also be phrased for the Eberlein convolution.

\begin{prop}\label{prop3}
Let $\mu, \nu \in \cM^\infty(G)$ be such that $\mu \ebeA \nu$ is well defined.
\begin{itemize}
  \item[(a)] If $\mu \in \cM\mathtt{ap}_{\cA}(G)$, then $\mu \ebeA \nu \in \SAP(G)$.
  \item[(b)] If $\nu \in \cM\mathtt{ap}_{-\cA}(G)$, then $\mu \ebeA \nu \in \SAP(G)$. \qed
\end{itemize}
\end{prop}

\begin{remark}
Theorem~\ref{thm tebe is SAP} (b) and Proposition~\ref{prop3} (b) emphasize the advantages of working with the twisted Eberlein convolution
over its standard version.   \exend
\end{remark}

Next, we give an estimate for the norms of the twisted Eberlein convolution, which we will need later. First let us state the following results.

\begin{lemma}  \label{lem:2}
Let $\mu,\nu\in \mc{M}^{\infty}(G)$, and let
$\mc{A}=(A_n)_{n\in\NN}$ be a van Hove sequence such that $
\leftbrac \mu ,  \nu \rightbrac_{\cA}$ exists. Then, for each
$\varphi \in \Cc(G)$ and all $t\in G$, the following limit exists
and
\[
\lim_{n\to\infty} \frac{1}{|A_n|} \int_{A_n}
\overline{(\varphi*\nu)(s-t)}\ \dd \mu(s) =
(\widetilde{\varphi}*\leftbrac \mu ,\nu  \rightbrac_{\cA} )(t) \ts.
\]
\end{lemma}
\begin{proof}
Fix $\varphi \in \Cc(G)$ and $t \in G$.
Let $\eps >0$.
Note first that
\[
(\widetilde{\varphi}* \leftbrac \mu ,  \nu \rightbrac_{\cA} )(t)=
\leftbrac \mu , \nu \rightbrac_{\cA} (\overline{ T_t \varphi}) \ts.
\]
Therefore, by Corollary~\ref{coro:second def Ebe}, there exists some $N$ such that
\[
\lim_{n\to\infty} \frac{1}{|A_n|} ( \mu|_{A_n}  * \widetilde{ \nu})
(\overline{ T_t \varphi}) = (\widetilde{\varphi}* \leftbrac \mu ,
\nu \rightbrac_{\cA}  )(t)
\]
for all $n>N$. Next,  we have
\begin{align*}
\frac{1}{|A_n|} ( \mu|_{A_n} * \widetilde{\nu}) ( \overline{ T_t \varphi})
    &= \int_{G} \int_{G} \overline{T_t \varphi (r+s)}\ \dd \widetilde{\nu}(r)\ \dd \mu|_{A_n}(s)
     = \int_{A_n} \int_{G} \overline{\varphi (r+s-t)}\ \dd \widetilde{\nu}(r)\ \dd \mu(s)  \\
    &=  \int_{A_n} \overline{\int_{G} \varphi(-r+s-t)\ \dd\nu(r)}\ \dd \mu(s)
     = \int_{A_n} \overline{\int_{G} \varphi(-r+s-t )\ \dd\nu(r)}\ \dd \mu(s) \\
    &=  \int_{A_n} \overline{(\varphi*\nu)(s-t)}\ \dd \mu(s)
\end{align*}
for all $n\in\NN$. The claim follows.
\end{proof}

 We can now give the following estimates on the norm of the twisted Eberlein convolution, which allows us study continuity and almost periodicity with respect to the corresponding topology.

\begin{lemma}
Let $\nu\in \mc{M}^{\infty}(G)$, and let $\mc{A}=(A_n)_{n\in\NN}$ be a van Hove sequence. Let $K, K'\subseteq G$ be a compact subsets such that $K \subseteq (K')^{\circ}$. Then, there exists a constant $C>0$ such that, for all $\nu \in {\mathcal TEB}_{\cA}(\mu)$ and all $\varphi \in \Cc(G)$, we have
\[
\| \leftbrac \mu ,\nu  \rightbrac_{\cA} \|_{K} \leqslant C\ts
\|\nu\|_{K'}  \qquad  \text{ and } \qquad \| \leftbrac \mu ,  \nu
\rightbrac_{\cA}\|_{\varphi} \leqslant C\ts \| \nu
\|_{\widetilde{\varphi}} \ts.
\]
\end{lemma}
\begin{proof}
By \cite[Lemma~1.1.(2)]{Martin2}, there exists a constant $C$ such that
\[
\frac{|\mu|(A_n)}{|A_n|} < C \qquad \text{ for all $n$} \ts.
\]
Pick some $\varphi \in \Cc(G)$ such that $1_{K} \leqslant \varphi \leqslant 1_{K'}$.
Then,
\begin{align*}
\| \leftbrac \mu , \nu \rightbrac_{\cA} \|_K
   &= \sup_{t\in G} | \leftbrac \mu , \nu \rightbrac_{\cA} |(t+K)  \\
   &\leqslant \sup_{t\in G}\ts \limsup_{n\to\infty} \frac{1}{|A_n|} \int_G\int_G \varphi(x+y-t )
      \ \dd|\mu|(x)\ \dd|\widetilde{\nu}|_{A_n}|(y)     \\
   &\leqslant \sup_{t\in G}\ts \limsup_{n\to\infty} \frac{1}{|A_n|} \int_G|\mu|(t+K')
     \ \dd|\widetilde{\nu}|_{A_n}|(y)   \\
   &\leqslant \|\mu\|_{K'} \limsup_{n\to\infty} \frac{1}{|A_n|} \int_{G} \ \dd|\widetilde{\nu}|_{A_n}|(y)   \\
   &\leqslant C\ts \|\mu\|_{K'}.
\end{align*}
This proves the first inequality.

\medskip

Next, by Lemma~\ref{lem:2}, we have
\begin{align*}
\| \leftbrac \mu , \nu \rightbrac_{\cA} \|_{\varphi}
    &= \sup_{t \in G} |(\widetilde{\widetilde{\varphi}}* \leftbrac \mu , \nu \rightbrac_{\cA}  )(t) | \\
    &= \sup_{t \in G} \left|\lim_{n\to\infty} \frac{1}{|A_n|} \int_{A_n} \overline{(\widetilde{\varphi}*\nu)
       (s-t)}\ \dd \mu(s) \right| \\
   &\leqslant \sup_{t \in G} \limsup_{n\to\infty} \frac{1}{|A_n|} \int_{A_n}\| \widetilde{\varphi}*\nu
      \|_\infty\  \dd |\mu|(s)  \\
    &\leqslant \sup_{t \in G} C \ts \| \nu \|_{\widetilde{\varphi}}= C \ts \| \nu \|_{\widetilde{\varphi}}  \ts. \qedhere
\end{align*}
\end{proof}

Next, we can show that the (twisted) Eberlein convolution is continuous with respect to the norm and product topology.

\begin{prop}
Let $K$ be any compact set with non-empty interior.
\begin{itemize}
\item[(a)] The mapping $\nu \mapsto  \leftbrac  \mu, \nu \rightbrac_{\cA} $ is continuous from $({\mathcal TEB}_{\cA}(\mu), \tau_{\operatorname{p}})$ to $(\cM^\infty(G), \tau_{\operatorname{p}})$.
\item[(b)] The mapping $\nu \mapsto  \leftbrac \mu , \nu \rightbrac_{\cA}$ is continuous from $({\mathcal TEB}_{\cA}(\mu), \| \cdot \|_{K})$ to $(\cM^\infty, \| \cdot \|_{K})$.
\end{itemize}
\end{prop}
\begin{proof}
(a) is obvious, while (b) follows from the fact that any two compact sets with non-empty interior define equivalent norms \cite{bm,SS1}.
\end{proof}

This result has some interesting consequences. First, let us show that the twisted Eberlein convolution preserves norm almost periodicity.

\begin{coro}
Let $\mu,\nu\in \mc{M}^{\infty}(G)$ and let $\mc{A}=(A_n)_{n\in\NN}$
be a van Hove sequence such that $ \leftbrac \mu,\nu
\rightbrac_{\cA}$ exists. If $\mu$ is norm almost periodic, then $
\leftbrac  \mu, \nu \rightbrac_{\cA}$ is norm almost periodic.  \qed
\end{coro}

 The previous results carry over for free to the Eberlein convolution.

\begin{coro}
Let $K$ be any compact set with non-empty interior.
\begin{itemize}
\item[(a)] The mapping $\nu \mapsto \mu \ebeA \nu$ is continuous from $({\mathcal EB}_{\cA}(\mu), \tau_{\operatorname{p}})$ to $(\cM^\infty, \tau_{\operatorname{p}})$.
\item[(b)] The mapping $\nu \mapsto \mu \ebeA \nu$ is continuous from $({\mathcal EB}_{\cA}(\mu), \| \cdot \|_{K})$ to $(\cM^\infty, \| \cdot \|_{K})$.  \qed
\end{itemize}
\end{coro}

\begin{coro}
Let $\mu,\nu\in \mc{M}^{\infty}(G)$ and let $\mc{A}=(A_n)_{n\in\NN}$ be a van Hove sequence such that $\mu\ebeA\nu$ exists. If $\mu$ is norm almost periodic, then $\mu\ebeA\nu$ is norm almost periodic.  \qed
\end{coro}

\medskip

We complete this section by looking at some examples.

\begin{remark}
The results we obtained in this section are optimal in the following sense.
\begin{enumerate}
\item[(a)] If $\mu,\nu\in\mathcal{M}^{\infty}(G)$, then $ \leftbrac \mu , \nu \rightbrac_{\cA} $ does not need to be strongly almost periodic.
\item[(b)] If $\mu,\nu\in\mathcal{M}^{\infty}(G)$ and $\mu$ is mean almost periodic, then $\ \leftbrac \mu , \nu \rightbrac_{\cA} $ does not need to be norm almost periodic.  \exend
\end{enumerate}
\end{remark}

To see this, let us consider two examples.

\begin{example}
The autocorrelation measure of the Bernoulli comb
\[
\omega_{B} = \sum_{m\in\ZZ} W_m\ts \delta_m,
\]
where $(W_m)_{m\in\ZZ}$ is a family of i.i.d. random variables (the latter taking the values $1$ and $-1$ with probabilities $p$ and $1-p$, where $0<p<1$) is almost surely given by
\[
\gamma_{\omega_B} = (2p-1)^2\ts \lm + 4p(1-p)\ts \delta_0
\]
\cite[Prop.~11.1]{TAO}, which is weakly but not strongly almost periodic.  \exend
\end{example}

\begin{example}
Let $\alpha>0$ be an irrational number. Define
\[
\mu:= \leftbrac \delta_{\ZZ}+\delta_{\alpha\ZZ} ,
\delta_{\ZZ}+\delta_{\alpha\ZZ} \rightbrac_{\cA}  = \delta_{\ZZ} +
\frac{2}{\alpha}\ts \lm + \frac{1}{\alpha}\ts \delta_{\alpha\ZZ},
\]
compare \cite[Ex. 8.10]{TAO}. A short calculation gives $\tau_t\mu=\delta_{t+\ZZ}+\frac{2}{\alpha}\ts \lm + \frac{1}{\alpha}\ts \delta_{t+\alpha\ZZ}$. Let $K$ be a compact subset of $\RR$. Then, we obtain
\begin{align*}
\|\tau_t\mu-\mu\|_K
    &= \sup_{x\in\RR}\ \Big|\delta_{t+\ZZ} - \delta_{\ZZ} + \frac{1}{\alpha}\ts \delta_{t+\alpha\ZZ} -\frac{1}{\alpha}
        \ts \delta_{\alpha\ZZ} \Big|(x+K)  \\
    &\geqslant \min \Big\{\frac{1}{\alpha},\ts 1,\ts 1+\frac{1}{\alpha},\ts |1-\frac{1}{\alpha}| \Big\}
\end{align*}
if $t\neq0$. Hence, even the twisted Eberlein convolution of two strongly almost periodic measures does not need to be norm almost periodic.  \exend
\end{example}

 Next, we construct an example of an absolutely continuous measure $\mu=f \lm$ with continuous density function such that $ \leftbrac \mu , \mu \rightbrac_{\cA}$ is a pure point measure.

\begin{example}
Let $(f_n)_{n\in\NN}$ be a sequence in $\Cc(\RR)$ with the following properties:
\begin{itemize}
\item{} $\supp(f_n) \subseteq [-\frac{1}{n}, \frac{1}{n}]$ and $\supp(f_0) \subseteq [-1,1]$.
\item{} $f_n \geqslant 0$.
\item{} $\int f_n(t)\ts \dd t =1$.
\end{itemize}
Define
\[
f= \sum_{n \in \ZZ} \delta_n*f_{|n|}  \qquad  \text{ and } \qquad \mu= f \lm  \ts.
\]
Then, for $\cA = ([-n,n])_{n\in\NN}$, the convolution $\mu \ebeA \mu$ exists and
\[
 \leftbrac \mu , \mu \rightbrac_{\cA}  = \delta_\ZZ \ts.
\]

To see this, note here that in the vague topology we have
\[
\lim_{|n| \to \infty} (f_n \theta_G- \delta_n) =0 \ts.
\]
Then, by \cite[Prop.~41]{SS4} we have $\mu -\delta_{\ZZ} \in  \WAP_{0}(G)$.
In particular, $\mu \in \WAP(\RR)$ has the Eberlein decomposition
\[
  \mu_{\mathsf{s}} =\delta_{\ZZ}  \ts, \qquad
  \mu_{0} =\mu- \delta_{\ZZ} \ts.
\]
Now, by \cite[Lemma~7.10]{LS2}, $\mu$ and $\mu_{\mathsf{s}}$ have the same autocorrelation, which proves the claim.
\exend
\end{example}

\bigskip

At the end of this section, let us talk about Besicovitch almost periodic measures. The following results are obtained from subsection~\ref{subsec:bes} via standard properties of the convolution
\[
\mathcal{B}ap(G) \cap \cM^{\infty}(G) \to Bap(G)\cap L^{\infty}(G) \ts, \qquad \mu \mapsto \mu *\phi \ts,
\]
for $\phi\in\Cc(G)$.
Since the proofs are straightforward and standard, we will omit them.

\begin{lemma}
Let $\mu\in \cM^{\infty}(G)$, and let $\chi\in\widehat{G}$. The following assertions are equivalent.
\begin{enumerate}
\item[(i)] $a_{\chi}^{\cA}(\mu)$ exists.
\item[(ii)] $\leftbrac \mu  , \chi \rightbrac_{\cA}$ exists.
\item[(iii)] There is a $t\in G$ such that
\[
\leftbrac \mu  , \chi \rightbrac_{\cA}(t)= \lim_{n\to\infty} \frac{1}{|A_n|} \int_{A_n}\overline{\chi(s-t)}\ts  \dd \mu(s)
\]
exists.
\end{enumerate}
In that case, one has $\leftbrac \mu  , \chi \rightbrac_{\cA}(t) =
\chi(t)\ts a_{\chi}^{\cA}(\mu)$.  \qed
\end{lemma}

\begin{coro}
Let $\mu\in\cM^{\infty}(G)$. Then, $a_{\chi}^{\cA}(\mu)$ exists for
all $\chi\in\widehat{G}$ if and only if $\leftbrac \mu  , P
\rightbrac_{\cA}$ exists for all trigonometric polynomials
$P=\sum_{k=1}^n a_k\chi_k$.

In that case, one has $\leftbrac \mu  , P \rightbrac_{\cA}=
\sum_{k=1}^n \overline{a_k}\ts a_{\chi_k}^{\cA}(\mu)\ts \chi_k$. \qed
\end{coro}

 Exactly as for functions, the twisted Eberlein convolution exists for Besicovitch almost periodic measures, is strongly almost periodic, and satisfies the generalized CPP.

\begin{prop}
Let $\nu\in\mathcal{B}ap(G)\cap\cM^{\infty}(G)$, and let $\mu\in\cM^{\infty}(G)$ be such that $a_{\chi}^{\cA}(\mu)$ exists for all $\chi\in\widehat{G}$. Then,
\begin{enumerate}
\item[(a)] $\leftbrac \mu  , \nu \rightbrac_{\cA}$ exists.
\item[(b)] $\leftbrac \mu  , \nu \rightbrac_{\cA}\in\mathcal{SAP}(G)$.
\item[(c)] $a_{\chi}^{\cA}(\leftbrac \mu  , \nu \rightbrac_{\cA}) = a_{\chi}^{\cA}(\mu) \ts \overline{a_{\chi}^{\cA}(\nu)}$.   \qed
\end{enumerate}
\end{prop}

In particular, we get that the twisted Eberlein convolution of two translation bounded Besicovitch almost periodic measures exists, is strongly almost periodic,
and satisfies the strong CPP.

\begin{coro}
Let $\mu,\nu \in\mathcal{B}ap(G)\cap\cM^{\infty}(G)$. Then,
\begin{enumerate}
\item[(a)] $\leftbrac \mu  , \nu \rightbrac_{\cA}$ exists.
\item[(b)] $\leftbrac \mu  , \nu \rightbrac_{\cA}\in\mathcal{SAP}(G)$.
\item[(c)] $a_{\chi}^{\cA}(\leftbrac \mu  , \nu \rightbrac_{\cA}) = a_{\chi}^{\cA}(\mu) \ts \overline{a_{\chi}^{\cA}(\nu)}$.  \qed
\end{enumerate}
\end{coro}

%\begin{theorem}
%Let $\mu\in\cM^{\infty}(G)$. Then, the following assertions are equivalent.
%\begin{enumerate}
%\item[(i)] $\mu\in\mathcal{B}ap(G)$.
%\item[(ii)] One has
%   \begin{enumerate}
%   \item[$\bullet$] $a_{\chi}(\mu)$ exists for all $\chi\in\widehat{G}$,
%   \item[$\bullet$] $\leftbrac \mu  , \nu \rightbrac_{\cA}$ exists and $\leftbrac \mu  , \nu \rightbrac_{\cA}\in\mathcal{SAP}(G)$ for all $\nu\in\cM^{\infty}(G)$,
%   \item[$\bullet$] $a_{\chi}^{\cA}(\leftbrac \mu  , \nu \rightbrac_{\cA}) = a_{\chi}^{\cA}(\mu) \ts \overline{a_{\chi}^{\cA}(\nu)}$ for all $\nu\in\cM^{\infty}(G)$.
%   \end{enumerate}
%\item[(iii)] One has
%   \begin{enumerate}
%   \item[$\bullet$] $a_{\chi}(\mu)$ exists for all $\chi\in\widehat{G}$,
%   \item[$\bullet$] $\leftbrac \mu  , \mu \rightbrac_{\cA}$ exists and $\leftbrac \mu  , \mu \rightbrac_{\cA}\in\mathcal{SAP}(G)$,
%   \item[$\bullet$] $a_{\chi}^{\cA}(\leftbrac \mu  , \mu \rightbrac_{\cA}) = |a_{\chi}^{\cA}(\mu)|^2 $.
%   \end{enumerate}
%\end{enumerate}
%\end{theorem}

\section{Fourier transformability of (twisted) Eberlein convolution}

We complete the paper by showing that the twisted Eberlein decomposition of translation bounded measures is a linear combination of positive definite measures. In particular it is Fourier transformable and weakly almost periodic, and listing some consequences of this.

\begin{prop}\label{ebe is comb pd}
Let $\mu, \nu \in \cM^\infty(G)$. Then $ \leftbrac \mu , \nu \rightbrac_{\cA}$ is a linear combination of (at most four) positive definite measures.
\end{prop}
\begin{proof}

 Consider the polarisation identity \cite[p.~244]{MoSt}
\begin{align*}
\mu|_{A_n}*\widetilde{\nu}|_{-A_n}
   &= \frac{1}{4} \bigg[ (\mu|_{A_n}+\nu|_{A_n})*\stackon[-8pt]{(\mu|_{A_n}
       +\nu|_{A_n})}{\vstretch{1.5}{\hstretch{2.4}{\widetilde{\phantom{\;\;\;\;\;\;\;\;}}}}}
       -(\mu|_{A_n}-\nu|_{A_n})*\stackon[-8pt]{(\mu|_{A_n}
       -\nu|_{A_n})}{\vstretch{1.5}{\hstretch{2.4}{\widetilde{\phantom{\;\;\;\;\;\;\;\;}}}}} \\
   &\phantom{===}{} +\im\ts (\mu|_{A_n}+\im\nu|_{A_n})*\stackon[-8pt]{(\mu|_{A_n}
       +\im\nu|_{A_n})}{\vstretch{1.5}{\hstretch{2.4}{\widetilde{\phantom{\;\;\;\;\;\;\;\;}}}}}
       -\im\ts (\mu|_{A_n}-\im\nu|_{A_n})*\stackon[-8pt]{(\mu|_{A_n}
       -\im\nu|_{A_n})}{\vstretch{1.5}{\hstretch{2.4}{\widetilde{\phantom{\;\;\;\;\;\;\;\;}}}}}
         \bigg].
\end{align*}
Now, by Lemma~\ref{lem:tebe exists subseq}, there exists a subsequence $\mc{B}=(B_n)_{n\in\NN}$ of $\mathcal{A}$ such that the four limits
\[
\lim_{n\to\infty} \frac{1}{|B_n|} (\mu|_{B_n}+\nu|_{B_n})*\stackon[-8pt]{(\mu|_{B_n}
       +\nu|_{B_n})}{\vstretch{1.5}{\hstretch{2.4}{\widetilde{\phantom{\;\;\;\;\;\;\;\;}}}}}, \quad
       \lim_{n\to\infty} \frac{1}{|B_n|} (\mu|_{B_n}-\nu|_{B_n})*\stackon[-8pt]{(\mu|_{B_n}
       -\nu|_{B_n})}{\vstretch{1.5}{\hstretch{2.4}{\widetilde{\phantom{\;\;\;\;\;\;\;\;}}}}},
\]
\[
\lim_{n\to\infty} \frac{1}{|B_n|} (\mu|_{B_n}+\im\rho|_{B_n})*\stackon[-8pt]{(\mu|_{B_n}
       +\im\nu|_{B_n})}{\vstretch{1.5}{\hstretch{2.4}{\widetilde{\phantom{\;\;\;\;\;\;\;\;}}}}}, \quad
       \lim_{n\to\infty} \frac{1}{|B_n|} (\mu|_{B_n}-\im\nu|_{B_n})*\stackon[-8pt]{(\mu|_{B_n}
       -\im\nu|_{B_n})}{\vstretch{1.5}{\hstretch{2.4}{\widetilde{\phantom{\;\;\;\;\;\;\;\;}}}}}
\]
exist. Then, these measures are positive definite, and $\leftbrac
\mu  ,  \nu \rightbrac_{\cB}$ is a linear combination of these four
measures. Finally, since $\cB$ is a subsequence of $\cA$, and $
\leftbrac  \mu, \nu \rightbrac_{\cA}$ exists, we have
\[
 \leftbrac \mu , \nu \rightbrac_{\cA}  = \leftbrac  \mu, \nu \rightbrac_{\cB}  \ts.  \qedhere
\]
\end{proof}

By combining Proposition~\ref{ebe is comb pd} with \cite[Thm.~4.11.5 and Cor.~4.11.6]{MoSt}, we obtain the following result (compare \cite[Lemma~1]{BG}).

\begin{coro}\label{cor3}
Let $\mu,\nu\in\mc{M}^{\infty}(G)$, and let $\mc{A}=(A_n)_{n\in\NN}$
be a van Hove sequence such that $ \leftbrac \mu ,  \nu
\rightbrac_{\cA}$ exists. Then,
\begin{enumerate}
\item[(a)] $ \leftbrac \mu , \nu \rightbrac_{\cA} \in\WAP(G)$.
\item[(b)] $ \leftbrac \mu , \nu \rightbrac_{\cA} $ is Fourier transformable. \qed
\end{enumerate}
\end{coro}

As usual, we get a twin result for the Eberlein convolution.

\begin{coro}
Let $\mu,\nu\in\mc{M}^{\infty}(G)$, and let $\mc{A}=(A_n)_{n\in\NN}$ be a van Hove sequence such that $\mu\ebeA \nu$ exists. Then.
\begin{enumerate}
\item[(a)] $\mu \ebeA \nu$ is a linear combination of (at most 4) positive definite measures.
\item[(a)] $\mu\ebeA \nu\in\WAP(G)$.
\item[(b)] $\mu\ebeA \nu$ is Fourier transformable. \qed
\end{enumerate}
\end{coro}

We complete the paper by proving the following improvement of Corollary~\ref{cor2}.

\begin{coro}\label{cor1}
Let $f, g \in \Cu(G)$ be such that $\leftbrac
f , g \rightbrac_{\cA}$ exist. Then, $\leftbrac f , g
\rightbrac_{\cA}$ is weakly almost periodic.
\end{coro}
\begin{proof} Let $\mu=f \theta_G$ and $\nu= g \theta_G$. By Proposition~\ref{prop1}, $\leftbrac \mu, \nu \rightbrac_{\cA}$ exists and
\[
\leftbrac \mu, \nu \rightbrac_{\cA} =\leftbrac f, g \rightbrac_{\cA} \theta_G \,.
\]
Now, by Corollary~\ref{core tebe cu}, we have $\leftbrac f, g \rightbrac_{\cA} \in \Cu(G)$, while by Corollary~\ref{cor3} we have $\leftbrac \mu, \nu \rightbrac_{\cA} \in \WAP(G)$. Therefore,
$\leftbrac f, g \rightbrac_{\cA}\in WAP(G)$ by \cite[Prop.~4.10.5]{MoSt}.
\end{proof}

 \subsection*{Acknowledgments} This work was supported by the German Research Foundation (DFG), within the
CRC 1283 at Bielefeld University and via research grant 415818660
(TS), and by the Natural Sciences and Engineering Council of Canada
(NSERC), via grant 2020-00038 (NS).

\end{document}